\documentclass[10pt]{amsart}
\usepackage{amsmath}
\usepackage{amsfonts}
\usepackage{amssymb}
\usepackage{amsthm}
\usepackage{graphicx}
\usepackage{hyperref}

\def \B {\hat{B}}
\def \Ch {\hat{C}}
\def \C {{\mathcal{C}}}
\def \d {\hat{d}}
\def \de {\partial}
\def \G {\mathrm{G}}
\def \H {\mathcal{H}}
\def \HH {\textsc{{H}} }
\def \N {\mathbb{N}}
\def \O {\Omega}
\def \phi {\varphi}
\def \RN {\mathbb{R}^N}
\def \R {\mathbb{R}}
\def \l {\lambda}
\def \lgM {\log{(M)}}
\def \Dhk {D_k^h(z_0,\l)}
\def \Ohk {\O_k^h(z_0,\l)}
\def \Fhk {F_k^h}
\def \W {\mathcal{W}}
\def \qz {[z]}
\def \Orol {\O_\l^\rho}
\def \erolt {E_\l(\rho,\tau)}
\def \mrolt {m_\l(\rho,\tau)}

\newtheorem{theorem}{Theorem}[section]
\newtheorem{lemma}[theorem]{Lemma}
\newtheorem{proposition}[theorem]{Proposition}
\newtheorem{corollary}[theorem]{Corollary}

\newtheorem{remark}[theorem]{Remark}

\theoremstyle{definition}
\newtheorem{definition}[theorem]{Definition}

\def\stepone{\noindent{\it Step I. }}
\def\steponep{\noindent{\it Step $I^*$. }}
\def\steptwo{\noindent{\it Step II. }}
\def\steptwop{\noindent{\it Step $II^*$. }}
\def\stepthree{\noindent{\it Step III. }}
\def\stepfour{\noindent{\it Step IV. }}

\numberwithin{equation}{section}

\begin{document}

\title{Wiener-type tests from a two-sided Gaussian bound}

\author[E. Lanconelli]{Ermanno Lanconelli}
\address{Dipartimento di Matematica,
         Universit\`{a} degli Studi di Bologna,
         Piazza di Porta S. Donato, 5 - 40126 Bologna, Italy.
         }
 \email{ermanno.lanconelli@unibo.it}
\author[G. Tralli]{Giulio Tralli}
\address{Dipartimento di Matematica,
         Universit\`{a} degli Studi di Bologna,
         Piazza di Porta S. Donato, 5 - 40126 Bologna, Italy.
         }
 \email{giulio.tralli2@unibo.it}
\author[F. Uguzzoni]{Francesco Uguzzoni}
\address{Dipartimento di Matematica,
         Universit\`{a} degli Studi di Bologna,
         Piazza di Porta S. Donato, 5 - 40126 Bologna, Italy.        
         }
 \email{francesco.uguzzoni@unibo.it}

%\subjclass[2010]{35H10, 35K65, 31E05, 35H20, 31C15.}
%\keywords{Gaussian bounds, potential analysis, boundary behavior of PW-solutions, nondivergence H\"ormander operators, Wiener criterion.}

\date{}

\begin{abstract}
In this paper we are concerned with hypoelliptic diffusion operators $\H$. Our main aim is to show, with an axiomatic approach, that a Wiener-type test of $\H$-regularity of boundary points can be derived starting from the following basic assumptions: Gaussian bounds of the fundamental solution of $\H$ with respect to a distance satisfying doubling condition and segment property. As a main step towards this result, we establish some estimates at the boundary of the continuity modulus for the generalized Perron-Wiener solution to the relevant Dirichlet problem. The estimates involve Wiener-type series, with the capacities modeled on the Gaussian bounds. We finally prove boundary H\"older estimates of the solution under a suitable exterior cone-condition.
\end{abstract}
\maketitle

\section{Introduction}\label{intro}
Let us consider the following linear second order Partial Differential Operators
\begin{equation}\label{IHG}
\H=\sum_{i,j=1}^N q_{i,j}(z)\de^2_{x_i,x_j} +\sum_{k=1}^N q_{k}(z)\de_{x_k}
-\partial_t ,
 \end{equation}
in the strip of $\R^{N+1}$
$$ S=\{ z=(x,t)\, :\, x\in\RN,\ T_1<t<T_2\},\quad -\infty\le T_1<T_2\le \infty.$$
We assume the coefficients $q_{i,j}=q_{j,i}, q_k$
of class $C^\infty$, and the characteristic form
$$q_\H(z,\xi)=\sum_{i,j=1}^N q_{i,j}(z)\xi_i\xi_j,\quad \xi=(\xi_1,\ldots,\xi_N)\in\RN,$$
nonnegative definite and not totally degenerate, i.e.,
$q_\H(z,\cdot)\ge 0,\ q_\H(z,\cdot)\not\equiv 0$ for every $z\in S$. We also assume the {\em hypoellipticity} of $\H$ and of its adjoint $\H^*$, and the existence of a global {\em fundamental solution}
$$(z,\zeta )\mapsto \Gamma (z,\zeta )$$
smooth out of the diagonal of $S\times S$.

In \cite{LU} a deep Potential Analysis for $\H$ has been developed {\it only} assuming a {\it two sided Gaussian-type estimate} for $\Gamma$. Such analysis was mainly aimed to obtain regularity criteria and uniform boundary estimates for the Perron-Wiener-Brelot-Bauer (PWBB, in short) solution to the Dirichlet problem for $\H$ in terms of suitable series involving balayage potentials. Under the same assumptions, our objective here is to prove Wiener-type tests for $\H$ and to estimate the continuity modulus at the boundary of the PWBB-solution in terms of Wiener-type series, i.e. series involving $\H$-capacity of {\it ring-shaped} sets of $\Gamma$.

Before we state our main results we want to give a precise description of our assumptions and to recall some notations and results from \cite{LU}. First of all, when we say that $\Gamma$ is a fundamental solution for $\H$ we mean
\begin{enumerate}
\item[(i)] $\Gamma (\cdot,\zeta )\in  L^1_{\rm {loc}}(S)$ and $\H(\Gamma (\cdot,\zeta ))=-\delta _\zeta$, the Dirac measure at
 $\{\zeta\}$, for every $\zeta \in S$;
\item[(ii)] for every compactly supported continuous function $\phi$ on $\RN$ and for every $x_0\in\RN$, $\tau\in]T_1,T_2[$, we have
\begin{equation}\label{H1}
\int_{\RN} \Gamma(x,t,\xi,\tau)\,\varphi(\xi)\,d\xi\to\phi(x_0),\qquad\text{as $(x,t)\to(x_0,\tau)$, $t>\tau$.}
\end{equation}
\end{enumerate}
\noindent
Given a metric $d:\R^N\times\R^N\rightarrow\R$, we call $d$-Gaussian (of exponent $a>0$) any function
$$\G^{(d)}_a(z,\zeta)=\G^{(d)}_a(x,t,\xi,\tau)=
\begin{cases}
0 & \text{ if }t\le\tau,\\
\frac{1}{|B_d(x,\sqrt{t-\tau})|} \exp\left(-a \frac{d(x,\xi)^2}{t-\tau}\right) &  \text{ if }t>\tau.
\end{cases}
$$
Hereafter, if $A\subseteq\RN$ ($A\subseteq\mathbb{R}^{N+1}$), $|A|$ denotes the $N$-dimensional ($(N+1)$-dimensional) Lebesgue measure of $A$. Moreover, we denote the $d$-ball of center $x$ and radius $r>0$ as
$$B_d(x,r)=B(x,r)=\{y\in\RN\,:\,d(x,y)<r\}.$$
Then, our crucial \emph{axiomatic} assumption is the existence of a distance $d$ in $\R^N$ such that the following Gaussian estimates for $\Gamma$ hold
\vskip 0.35cm
\begin{enumerate}\label{bounds}
\item[(H)] $\qquad\qquad\qquad\frac{1}{\Lambda} \G^{(d)}_{b_0}(z,\zeta)\le \Gamma(z,\zeta)\le\Lambda\G^{(d)}_{a_0}(z,\zeta),\quad \forall z,\zeta\in S,$
\end{enumerate}\vskip 0.35cm
for suitable positive constants $a_0$, $b_0$, and $\Lambda$. Throughout the paper we keep such a distance $d$ fixed, and we will simply write $\G_a$ instead of $\G^{(d)}_a$. We shall make the following assumptions on the metric space ($\R^N, d$):
\begin{enumerate}
\item[(D1)]\label{diuno} The $d$-topology is the Euclidean topology. Moreover $(\RN,d)$ is complete and, for every fixed $x\in\RN$, $d(x,\xi)\to\infty$ if (and only if) $\xi\to\infty$ with respect to the usual Euclidean norm.
\item[(D2)]\label{didue} $(\RN,d)$ is a {\em doubling metric space} w.r.t. the Lebesgue measure, i.e. there exists a constant $c_d>1$ such that
$$ |B(x,2r)|\le c_d |B(x,r)|, \quad \forall x\in\RN,\ \forall r>0.$$
We will always denote by $Q=\log_2{c_d}$ the relative homogeneous dimension.
\item[(D3)]\label{ditre} $(\RN,d)$ has the {\em segment property}, i.e., for every $x,y\in\RN$ there exists a continuous path $\gamma: [0,1]\to\RN$ such that $\gamma(0)=x$, $\gamma(1)=y$ and
$$d(x,y)=d(x,\gamma(t))+d(\gamma(t),y)\quad\forall t\in [0,1].$$
\end{enumerate}
Given an operator $\H$ satisfying (\hyperref[bounds]{H}) w.r.t. a metric $d$ verifying  (\hyperref[diuno]{D1})--(\hyperref[ditre]{D3}), we set
\begin{equation}\label{dipacca}
|\H|=\Lambda+a_0^{-1}+b_0+c_d.
\end{equation}
The operator $\H$ endows the strip $S$ with a structure of $\beta$-harmonic space satisfying the Doob convergence property, see \cite[Theorem 3.9]{LU}. As a consequence, for any bounded open set $\O$ with $\overline{\O}\subseteq S$, the Dirichlet problem
$$\begin{cases}
\H u= 0 \text{ in }\Omega,  \\
u|_{\de\Omega}=\phi
\end{cases}$$
has a generalized solution $H_\phi^\Omega$, in the Perron-Wiener sense, for every continuous function $\phi:\de\Omega\rightarrow\R$. A point $z_0\in\de\O$ is called $\H$-regular if $\lim_{z\rightarrow z_0}{H_\phi^\Omega(z)}=\phi(z_0)$ for every $\phi\in C(\de\O)$. The main result of this paper is the following Wiener-type test for the $\H$-regularity of the boundary points of $\O$.

\begin{theorem}\label{mmmain} Let $z_0=(x_0,t_0)\in\de\O$, and $\l\in]0,1[$.
\begin{itemize}
\item[(i)] If there exists $0<a\leq a_0$ and $b>b_0$ such that
\begin{equation}\label{anonb}
\sum_{h,k=1}^{+\infty}{\frac{\C_a\left(\Ohk\right)}{\left|B\left(x_0,\sqrt{\l^k}\right)\right|}\l^{bh}}=+\infty
\end{equation}
then the point $z_0$ is $\H$-regular.
\item[(ii)] If the point $z_0$ is $\H$-regular, then
\begin{equation}\label{bnona}
\sum_{h,k=1}^{+\infty}{\frac{\C_b\left(\Ohk\right)}{\left|B\left(x_0,\sqrt{\l^k}\right)\right|}\l^{ah}}=+\infty
\end{equation}
for every $b\geq b_0$ and $0<a\leq a_0$.
\end{itemize}
\end{theorem}
We postpone to the end of the Introduction (Subsection \ref{substo}) some comments on this theorem, along with an historical overview and explicit examples of operators to which our results apply.\\
In the above theorem we have denoted, for $h,k\in\N$,
\begin{eqnarray}\label{omhktilde}
\Ohk&=&\left\{\zeta=(\xi,\tau)\in S\smallsetminus\O\,:\,\l^{k+1}\leq t_0-\tau\leq\l^k,\vphantom{\left(\frac{1}{\l}\right)^h}\right.\\
&\,&\quad\left.\left(\frac{1}{\l}\right)^{h-1}\hspace{-0.15cm}\leq\exp{\left(\frac{d^2(x_0,\xi)}{t_0-\tau}\right)}\leq\left(\frac{1}{\l}\right)^h,\,\,\hat{d}(z_0,\zeta)\leq\sqrt{\l}\right\}.\nonumber
\end{eqnarray}

\vspace{-0.3cm}

\begin{figure}[htbp]
\begin{center}
\includegraphics{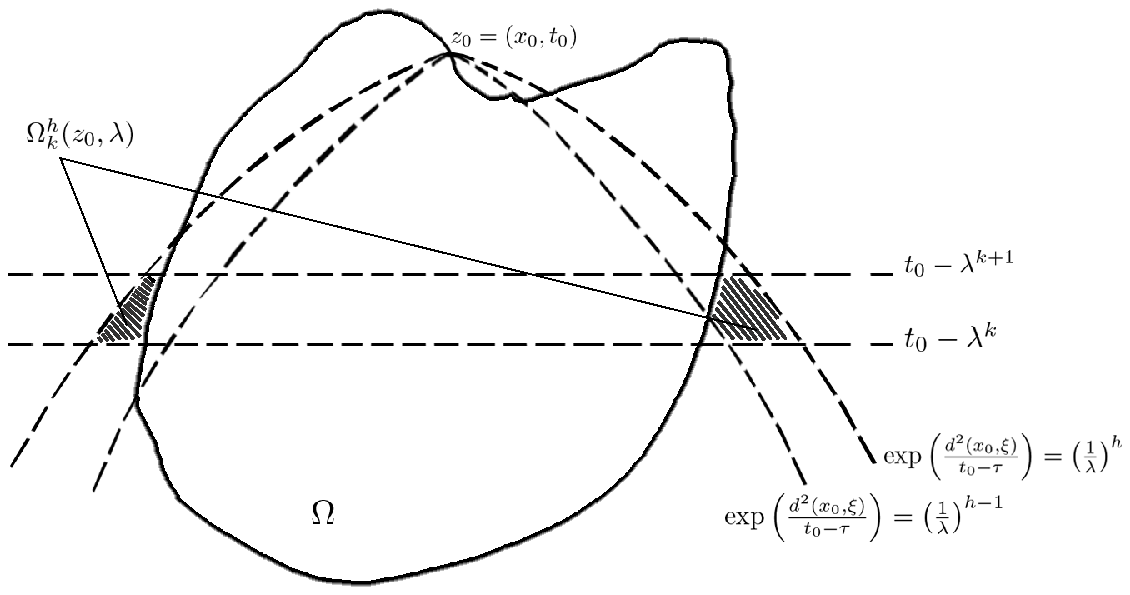}
\end{center}
\end{figure}

Moreover, for a compact set $F\subset S$, $\C_a(F)$ stands for the capacity of $F$ with respect to the $d$-Gaussian kernel of exponent $a$ (see Section \ref{cap} for the classical definition).
We have also set
$$\d(z,\zeta)=(d(x,\xi)^4+(t-\tau)^2)^{\tfrac{1}{4}},\quad  z=(x,t),\,\zeta=(\xi,\tau)\in S.$$
We shall call $\d$ the {\em parabolic counterpart} of $d$. The relative {\em parabolic balls} are
$$\B(z,r)=\{\zeta\in S\,:\, \d(z,\zeta)<r\},\quad z\in S,\ r>0.$$

As a main step in the proof of \eqref{anonb}, in Theorem \ref{onedotwo} below we establish an estimate of the continuity modulus of $H_\phi^\Omega$ at the boundary points of $\O$ in terms of Wiener-type series modeled on the $d$-Gaussian functions appearing in (\hyperref[bounds]{H}). To explain this estimate we need some more recalls from \cite{LU}. For $l\in\N$ and $\l\in]0,1[$, we denote by $V_l$ the $\H$-balayage potential of the set $\left(\B(z_0,\l^{\frac{l}{2}})\cap\{t\leq t_0\}\right)\smallsetminus\O$. We know that $0\leq V_l\leq 1$. For any $\rho\in]0,1[$, the function
\begin{equation}\label{vudoppia}
 \W=\W_{\rho}=\sum_{l=1}^\infty \rho^l(1-V_l)
\end{equation}
is what we call a {\em{$\H$-Wiener function for $\Omega$ at
 $z_0$}}. This function can be used to characterize the $\H$-regularity of the boundary point $z_0$. We have indeed (see \cite[Theorem 5.4]{LU})
\begin{equation}\label{charw}
\lim_{z\rightarrow z_0}{H_\phi^\Omega(z)}=\phi(z_0)\,\,\,\mbox{for every }\phi\in C(\de\O)\quad\mbox{if and only if }\quad \W(z)\rightarrow 0\,\,\,\mbox{as }z\rightarrow z_0.
\end{equation}
An even stronger result holds true: the continuity modulus of $H_\phi^\Omega$ at $z_0$ can be estimated only in terms of $\W$ and of the continuity modulus of the boundary data $\phi$. In fact, we have the following
\begin{equation}\label{5.9}
  |H_\phi^\Omega(z)-\phi(z_0)|\leq\tilde{\phi}(z_0,\W(z))
  \qquad\forall
   z\in\O,
 \end{equation}
where $\tilde{\phi}(z_0,s)$ is a suitable function, depending on $\phi$, which is monotone increasing in $s$ (we refer the reader to \cite[Theorem 5.2 and Remark 5.3]{LU} for the details). In this paper we show that the Wiener function $\W$, and thus $|H_\phi^\Omega(z)-\phi(z_0)|$, can be estimated in terms of Wiener-type series.
\begin{theorem}\label{onedotwo}
Let $a_0$ and $b_0$ be the positive constants in (\hyperref[bounds]{H}), and $\l\in]0,1[$. For every $0<a\leq a_0$ and $b>b_0$ there exist positive constants $C$ and $\rho_0$ (only depending on $a,b,|\H|,\l$) such that
\begin{equation}\label{wrhos}
\W_{\rho}(z)\leq C\exp{\left(-\frac{1}{C}\sum_{\N\ni k\leq \frac{\log{{\hat{d}}^2(z_0,z)}}{\log{\l}}}\sum_{h=1}^{+\infty}\frac{\C_a\left(\Ohk\right)}{\left|B\left(x_0,\sqrt{\l^k}\right)\right|}\l^{bh}\right)}\qquad\mbox{for every }z\in S,
\end{equation}
for any $0<\rho\leq\rho_0$.
\end{theorem}

From \eqref{wrhos} we can derive an integral estimate involving the Lebesgue measures of the following sections
$$\erolt=\left\{x\in\R^N\,:\,z=(x,\tau)\in S\smallsetminus\O,\,\d(z_0,z)\leq \sqrt{\l},\quad\exp{\left(\frac{d^2(x_0,x)}{t_0-\tau}\right)}\leq\rho\right\}.$$

\begin{theorem}\label{intnotcap}  For any $\l\in]0,1[$ and $b>b_0$ there exist positive constants $C$ and $\rho_0$ (only depending on $|\H|,\l,b$) such that
\begin{equation}\label{unodici}
\W_{\rho}(z)\leq C\exp{\left(-\frac{1}{C}\int_{\d^2(z_0,z)}^{\l}{\int_1^{+\infty}{\frac{\left|E_\l(\rho,t_0-\eta)\right|}{\left|B\left(x_0,\sqrt{\eta}\right)\right|}}\frac{d\rho}{\rho^{1+b}}\frac{d\eta}{\eta}}\right)}\qquad\mbox{for every }z\in S,
\end{equation}
for any $0<\rho\leq\rho_0$. In particular, $z_0$ is $\H$-regular if
$$\int_{0}^{\l}{\int_1^{+\infty}{\frac{\left|E_\l(\rho,t_0-\eta)\right|}{\left|B\left(x_0,\sqrt{\eta}\right)\right|}}\frac{d\rho}{\rho^{1+b}}\frac{d\eta}{\eta}}=+\infty$$
for some $b>b_0$.
\end{theorem}

Finally, from Theorem \ref{intnotcap}, we can obtain a H\"older estimate of the solutions at the boundary points satisfying an exterior $d$-cone condition. We explicitly remark that, under such geometrical condition, the series in \eqref{anonb} diverges for any $b$, ensuring the regularity of $z_0$ (as we already know by \cite{LU}).

\begin{theorem}\label{conooo}
Assume the exterior $d$-cone condition \eqref{4.4} holds at $z_0$. Let $\phi\in C(\de\O,\R)$ be such that
$[\phi]_{z_0,\delta}=\sup_{\rho>0}{\sup_{\d(z,z_0)\leq \rho}{\frac{|\phi(z)-\phi(z_0)|}{\rho^\delta}}}<\infty$
for some $\delta>0$. Then, there exist $0<\alpha_0\leq1$ and $c>0$ only depending on $|\H|, d, \delta, \O$, and the constants $M_0, r_0, \theta$ in the $d$-cone condition \eqref{4.4} such that
\begin{equation}\label{hphiquad}
|H_\phi^\Omega(z)-\phi(z_0)|\le c[\phi]_{z_0,\delta}\left(\d(z_0,z)\right)^{\alpha_0}
\end{equation}
for all $z\in\O$.
\end{theorem}

We would like to emphasize that the estimates \eqref{5.9}-\eqref{wrhos}-\eqref{unodici}-\eqref{hphiquad} depend on the operator $\H$ only through the constant $|\H|$ in \eqref{dipacca}. This allows to extend our results to operators with non-smooth coefficients. See also Subsection \ref{substo} below.

This paper is organized as follows. In Section \ref{cap} we recall the different notions of capacities we are going to exploit, i.e. the balayage capacity and the capacities with respect to $\Gamma$ and to the Gaussian kernels. We show how they compare each other and with the Lebesgue measure. In Section \ref{GL} we establish Gaussian bounds for the Green kernels on suitable cylinders. Then, we use them to prove a couple of very technical but powerful lemmas (Lemma \ref{CL1} and Lemma \ref{CL2}). They deal with an estimate of the balayage potential of some compact set by a term involving Green-equilibrium potentials. These lemmas will be crucial in Section \ref{ZSchD}, where we are going to complete the proofs of Theorem \ref{onedotwo} and Theorem \ref{mmmain}. We first establish Theorem \ref{onedotwo} and we derive from that Theorem \ref{mmmain}, part $(i)$. The proof of part $(ii)$ is obtained by bounding the balayage potential of $\Ohk$ and it closes Section \ref{ZSchD}. In Section \ref{IBCCond} we prove the integral estimate in Theorem \ref{intnotcap}. The $d$-cone condition allows us to bound further this integral term and to get at last Theorem \ref{conooo}.

\subsection{Some comments and historical notes}\label{substo} We would like to comment here on Theorem \ref{mmmain}. If $\G^{(e)}_a$ is the Gaussian function related to the Euclidean distance in $\R^N$, i.e.
$$\G^{(e)}_a(x,t,\xi,\tau)=\frac{1}{\omega_N(t-\tau)^{\frac{N}{2}}} \exp\left(-a \frac{\left|x-\xi\right|^2}{t-\tau}\right),\qquad\mbox{for }t>\tau,$$
then the $\G^{(e)}_a$-capacity of a compact set is independent of $a$. Precisely, for any $a,b>0$ there exists a constant $c=c(a,b)>0$ (independent of $F$), such that
\begin{equation}\label{cape}
\frac{1}{c}\C_a(F)\leq\C_b(F)\leq c\,\C_a(F)\qquad\mbox{for any compact set }F\subset \R^{N+1}.
\end{equation}
This not trivial result was proved in \cite[Proposizione 2]{L75}.\footnote{We do not know if \eqref{cape} holds for Gaussian kernels related to non-Euclidean distances.} As a consequence, in the Euclidean case, we can replace $\C_a$ and $\C_b$ in \eqref{anonb} and \eqref{bnona} with
$$\C=\C_{\frac{1}{4}}$$
that is the capacity related to the Heat kernel and to the Heat operator $\Delta-\de_t$. Hence, in the Euclidean case, the Wiener-type series in \eqref{anonb} and \eqref{bnona} take the form
\begin{equation}\label{lseriee}
\sum_{h,k=1}^{+\infty}{\frac{\C\left(\Ohk\right)}{\l^{\frac{kN}{2}}}\l^{\alpha h}}.
\end{equation}
By Theorem \ref{mmmain}, if the series in \eqref{lseriee} is divergent then  $z_0 \in \partial \Omega$ is $\left(\frac{1}{\beta}\Delta-\de_t\right)$-regular if\footnote{The fundamental solution of $\left(\frac{1}{\beta}\Delta-\de_t\right)$ is the Euclidean Gaussian $G^e_{\frac{\beta}{4}}$.} $0< \frac{\beta}{4}<\alpha$. Viceversa, if $z_0$ is $\left(\frac{1}{\beta}\Delta-\de_t\right)$-regular then the series \eqref{lseriee} is divergent
 for every $\alpha \leq \frac{\beta}{4}$. This result was first proved in \cite[Lemma 2.4 and Teorema A]{L75} (see also \cite{L77}). One of its consequences is the following one: if $z_0 \in \partial \Omega$ is $\left(\frac{1}{\beta}\Delta-\de_t\right)$-regular and $\gamma < \beta$,
 then $z_0$ is  $\left(\frac{1}{\gamma}\Delta-\de_t\right)$-regular. This result is sharp. Indeed, if $\gamma < \beta$, using the classical Petrowski's regularity criterion in \cite{Pe} (see also \cite[Theorem 8.1]{EK}), one can find an open set $\Omega$ and a point $z_0 \in \partial \Omega$ which is regular for $\left(\frac{1}{\gamma}\Delta-\de_t\right)$ and not regular for $\left(\frac{1}{\beta}\Delta-\de_t\right)$.

We would also like to recall here the celebrated Wiener test for the Heat equation. Define, for $k\in\N$,
$$\Omega'_{k}(z_0)=\left\{z\in \R^{N+1}\setminus\Omega\,:\, \l^{-k}\leq\Gamma^e(z_0,z)\le\l^{-(k+1)}\,\right\},$$
where $\Gamma^e=\G^{(e)}_{\frac{1}{4}}$ denotes the Heat kernel, i.e. the fundamental solution of $\Delta-\de_t$. Then,
$$z_0\mbox{ is }\left(\Delta-\de_t\right)\mbox{-regular}\qquad\mbox{if and only if}\qquad\sum_{k=1}^{+\infty}{\frac{\C\left(\Omega'_{k}(z_0)\right)}{\l^{k}}}=\infty.$$
The easy part of this criterion (its {\emph{only if}} part) was proved in \cite{L73}. The {\emph{if}} part is due to Evans and Gariepy in \cite{EG}. A necessary and sufficient condition of $\left(\Delta-\de_t\right)$-regularity in terms of Wiener-type series was previously proved by Landis in \cite{La}. Landis's criterion involves series of the type $\sum_k{v_k(z_0)}$, where $v_k$ is the Heat-equilibrium potential of $\{z\in\R^{N+1}\smallsetminus\O\,:\,\rho_k\leq\Gamma^e(z_0,z)\leq\rho_{k+1}\}$ being $\rho_k$ a certain sequence of positive real numbers such that $\frac{\rho_{k+1}}{\rho_k}\nearrow+\infty$. The first $\left(\Delta-\de_t\right)$-regularity criterion involving Heat-capacity and the level-rings of the fundamental solution appeared in literature in 1954, and it is due to Pini \cite{Pi}. Pini's result is related to the Heat equation in spatial dimension $N=1$, and gives a sufficient regularity criterion for particular open sets with continuous boundary. The Evans-Gariepy Wiener test was extended to parabolic operators with smooth variable coefficients by Garofalo and Lanconelli in \cite{GL}, and to parabolic operators with $C^1$-Dini continuous coefficients by Fabes-Garofalo-Lanconelli in \cite{FGL}.

Classical parabolic operators in divergence form, with merely measurable coefficients, are endowed with a fundamental solution satisfying the estimates (\hyperref[bounds]{H}) with respect to Euclidean Gaussians  $G^e_a$ (see \cite{Ar}, see also \cite{PE, lSC, Gr}). Then all our results apply to these equations: they were first proved in \cite{L75} and in \cite{L75p}.

{\it Degenerate parabolic} operators with H\"older continuous coefficients, of the kind
\begin{equation}\label{Hormander}
\sum_{i,j =1}^m a_{i,j}(x,t) X_i X_j + \sum_{k=1}^m a_k(x,t) X_k - \partial_t
\end{equation}
where $ {\mathcal{X}}=\{ X_1, \dots, X_m\}$ is a system of vector fields satisfying the celebrated H\"ormander rank condition, have a fundamental solution satisfying the estimates (\hyperref[bounds]{H}) with respect to $d$-Gaussian functions, where $d$ is the Carnot-Caratheodory distance related to ${\mathcal{X}}$.
The matrix $(a_{i,j})_{i,j= 1,\dots, m}$ is symmetric and uniformly strictly positive definite. Then, the operator in (\ref{Hormander}) can be suitably approximated with a sequence $(\H_j)_{j \in \N}$ of operators of the kind \eqref{IHG} with smooth coefficients and such that the sequence of constants $|\H_j|$ in \eqref{dipacca} has a finite upper bound (we directly refer to the papers \cite{BLU, BB, BBLU, LU, Ugu} for the details, see also \cite{JSC, SCS}). Then all our results in the present paper apply to the operators in \eqref{Hormander}.

In the stationary case, Wiener-type tests for second order degenerate-elliptic equations with underlying sub-Riemannian structures are well settled in literature, see the papers by Hueber \cite{Hu}, Hansen and Hueber \cite{HH}, Negrini and Scornazzani \cite{NS}; see also the very recent papers \cite{TU, Ug}, and the references therein. On the contrary, as far as we know only a few papers have been devoted to the Wiener test for evolution equations in sub-Riemannan settings: we mention the paper by Scornazzani \cite{Sc}, where a Wiener test of Landis-type for a Kolmogorov equation is proved, and the work \cite{GS} of Garofalo and Segala, in which the Wiener test for the Heat equation on the Heisenberg group is established. In these settings, more literature is available relating to the boundary behavior, in sufficiently regular domains, of {\emph{nonnegative}} solutions to evolution equations  (see e.g. the recent papers \cite{CNP1, CNP2, FGGMN, Mu}, and the references therein).

\section{Capacities}\label{cap}

We want to briefly recall here some classical notions of potential theory. They allow us to define and compare all the different capacities which play a big role for our scopes.

\vskip 0.5cm

For a given a compact set $F\subseteq S$, we put $$\Phi_F=\{v\in\overline{\HH}(S)\,:\,v\ge 0 \text{ in }S,\ v\ge 1 \text{ in }F\},$$
where $\overline{\HH}(S)$ is the set of $\H$-superharmonic functions in $S$. We also indicate $W_F=\inf\{v\,:\,v\in \Phi_F\}$. Then we can define\footnote{We agree to let
$$\liminf_{\zeta\to z}w(\zeta)=\sup_{V\in\mathcal{U}_z}(\inf_{V} w)$$
being $\mathcal{U}_z$ a basis of neighborhoods of $z$.} the ($\H$-){\em{balayage}} potential of $F$
$$V_F(z)=\liminf_{\zeta\to z}W_F(\zeta),\qquad z\in S.$$
We are going to denote by $\mu_F$ the Riesz-measure of $V_F$, and we let $\C_\H(F)=\mu_F(F)$.

\vskip 0.5cm

Let now $X$ be a Hausdorff locally compact topological space, and let $K:X\times X\rightarrow[0,+\infty]$ be a lower semicontinuous function. If in addition $K(\cdot,\zeta)\not\equiv0$ for any fixed $\zeta\in X$, we will say that $K$ is a kernel on $X$. Given a compact set $F\subseteq X$, we denote by $\mathcal{M}^+(F)$ the set of nonnegative Radon measures supported on $F$. Let us define
$$\C_K(F)=\sup{\left\{\mu(F)\,:\,\mu\in\mathcal{M}^+(F),\mbox{ and }\,K\ast\mu(z)=\int K(z,\zeta)d\mu(\zeta)\leq1\,\,\forall z\in X\right\}}.$$
Let us also denote by $\mathcal{F}(X)$ the collection of the compact subsets of $X$. The following statements are quite standard (see e.g. the classical paper by Fuglede \cite[Chapter 1, Section 2]{Fu}):
\begin{itemize}
\item[(i)] $\C_K(F)<\infty$ for any $F\in\mathcal{F}(X)$;
\item[(ii)] if $F_1,F_2\in\mathcal{F}(X)$ with $F_1\subseteq F_2$, then $\C_K(F_1)\leq\C_K(F_2)$;
\item[(iii)] if $K_1, K_2$ are kernel on $X$ such that $K_1\leq K_2$, then $\C_{K_1}(F)\geq\C_{K_2}(F)$ for every $F\in\mathcal{F}(X)$;
\item[(iv)] for every $F\in\mathcal{F}(X)$ there exists $\mu=\mu_K\in\mathcal{M}^+(F)$ with $K\ast\mu\leq 1$ in $X$ such that
$$\mu(F)=\C_K(F);$$
\item[(v)] if $F\subseteq\cup_{k\in\N}{F_k}$ with $F,F_k\in\mathcal{F}(X)$, then $\C_K(F)\leq\sum_{k\in\N}\C_K(F_k)$.
\end{itemize}
The measure $\mu_K$ will be called $K$-equilibrium measure of $F$, and the function $K\ast\mu_K$ will be called a $K$-equilibrium potential of $F$. In what follows we will exploit these notions mostly with the kernels $\Gamma$ and $\G_a$. We will always write $\C_a$ instead of $\C_{\G_a}$.

\vskip0.5cm

We now start to establish some capacitary estimates. The following will be exploited in Section \ref{ZSchD}.

\begin{proposition} Let $z_0=(x_0,t_0)\in S$, and let $\B(z_0,r)$ be such that $\B\left(z_0,(1+\theta)r\right)\subseteq S$, with $\theta>0$. Then, for any $a>0$, there exists a constant $C$ depending on $\theta$ such that
\begin{equation}\label{cipiuno}
\C_a\left(\vphantom{\sum^m}\overline{\B(z_0,r)}\right)\leq C\left|B(x_0,r)\right|.
\end{equation}
\end{proposition}
\begin{proof}
Let us put $z_r=(x_0,t_r)$, where $t_r=t_0+\left(1+\frac{\theta}{2}\right)^2r^2$. We note that $z_r\in S$. For every $z=(x,t)\in\overline{\B(z_0,r)}$ we have
$$t_r-t=t_r-t_0+t_0-t\geq \left(1+\frac{\theta}{2}\right)^2r^2-r^2\geq\theta r^2 \qquad\mbox{and}\qquad\frac{d^2(x,x_0)}{t_r-t}\leq\frac{r^2}{\theta r^2}=\frac{1}{\theta},$$
so that
$$\G_a(z_r,z)\geq\frac{1}{\left|B\left(x_0,\sqrt{\theta}r\right)\right|}\exp{\left(-\frac{a}{\theta}\right)} \geq\frac{1}{C}\frac{1}{\left|B(x_0, r)\right|}.$$
As a consequence, if $v$ and $\nu$ are, respectively, a $\G_a$-equilibrium potential and measure of $\overline{\B(z_0,r)}$, we have
$$1\geq v(z_r)=\int_{\overline{\B(z_0,r)}}\G_a(z_r,\zeta)\,d\nu(\zeta)\geq\frac{1}{C\left|B(x_0, r)\right|}\,\,\, \nu\left(\vphantom{\sum^m}\overline{\B(z_0,r)}\right)$$
and hence $\C_a\left(\vphantom{\sum^m}\overline{\B(z_0,r)}\right)=\nu\left(\vphantom{\sum^m}\overline{\B(z_0,r)}\right)\leq C\left|B(x_0, r)\right|$.
\end{proof}

\vskip0.5cm

Let $F$ be a compact set contained in $S$. We want to compare $\C_\H(F)$ and $\C_\Gamma(F)$. If $\mu_F$ is the balayage-measure of $F$, from the fact that $\Gamma\ast\mu_F\leq 1$ in $S$ (see \cite[Proposition 8.3]{LU}), we immediately get
$$\C_\H(F)\leq \C_\Gamma(F).$$
To prove the reverse inequality, let us denote by $\nu$ a $\Gamma$-equilibrium measure of $F$. Then $\Gamma\ast\nu\leq 1$ in $S$ so that, if $u\in\Phi_F$, we have
\begin{eqnarray*}
&\,&u-\Gamma\ast\nu\in\overline{\HH}(S\smallsetminus F),\quad \liminf_{S\smallsetminus F\ni z\rightarrow\zeta}{(u-\Gamma\ast\nu)(z)}\geq u(\zeta)-1\geq0\quad\forall\zeta\in\de F\\
&\qquad&\qquad\qquad\mbox{and }\liminf_{d(x,0)\rightarrow+\infty}{(u-\Gamma\ast\nu)(x,t)}\geq 0\quad\forall t\in]T_1,T_2[.
\end{eqnarray*}
The minimum principle (see \cite[Proposition 3.10]{LU}) implies $u\geq\Gamma\ast\nu$ in $S\smallsetminus F$. This inequality holds all over $S$ since $u\geq 1 \geq\Gamma\ast\nu $ on $F$. Thus $u\geq\Gamma\ast\nu$ for all $u\in\Phi_F$. As a consequence $W_F\geq\Gamma\ast\nu$ and hence
$$V_F(z)=\liminf_{\zeta\rightarrow z}W_F(\zeta)\geq\liminf_{\zeta\rightarrow z}\Gamma\ast\nu(\zeta)\geq\Gamma\ast\nu(z)\qquad\mbox{for all }z\in S.$$
From the fact that $V_F=\Gamma\ast\mu_F$ in $S\smallsetminus\de F$ (see \cite[Proposition 8.3]{LU}), we obtain
\begin{equation}\label{compot}
\Gamma\ast\nu\leq\Gamma\ast\mu_F\qquad\mbox{in }S\smallsetminus\de F.
\end{equation}
Now we need a lemma.
\begin{lemma}\label{commis}
Let $\mu_1, \mu_2\in \mathcal{M}^+(F)$ such that $\Gamma\ast\mu_1\leq\Gamma\ast\mu_2$ in $S\smallsetminus F$. Then
$$\mu_1(F)\leq c\,\mu_2(F).$$
\end{lemma}
\begin{proof}
Let $t\in]T_1,T_2[$ such that $F\subset\R^N\times]T_1,t[$. Then, for any $\xi\in\R^N$ and $\tau\in]T_1,t[$, we have
$$\frac{1}{\beta\Lambda}\leq\frac{1}{\Lambda}\int_{\R^N}\G_{b_0}(x,t,\xi,\tau)\,dx\leq\int_{\R^N}\Gamma(x,t,\xi,\tau)\,dx\leq\Lambda\int_{\R^N}\G_{a_0}(x,t,\xi,\tau)\,dx\leq\beta\Lambda$$
for some positive structural constant $\beta$ (see \cite[Proposition 2.2 and Proposition 2.4]{LU}). As a consequence
\begin{eqnarray*}
\mu_1(F)&=&\int_{F}d\mu_1\leq\beta\Lambda\int_{F}{\left(\int_{\R^N}\Gamma(x,t,\xi,\tau)\,dx\right)d\mu_1(\xi,\tau)}\\
&=&\beta\Lambda\int_{\R^N}{\left(\int_{F}\Gamma(x,t,\xi,\tau)\,d\mu_1(\xi,\tau)\right)dx}=\beta\Lambda\int_{\R^N}\Gamma\ast\mu_1(x,t)\,dx\\
&\leq&\beta\Lambda\int_{\R^N}\Gamma\ast\mu_2(x,t)\,dx=\beta\Lambda\int_{F}{\left(\int_{\R^N}\Gamma(x,t,\xi,\tau)\,dx\right)d\mu_2(\xi,\tau)}\\
&\leq&(\beta\Lambda)^2\int_{F}d\mu_2=c\,\mu_2(F).
\end{eqnarray*}
\end{proof}

The last lemma and inequality \eqref{compot} imply
$$\C_\Gamma(F)=\nu(F)\leq c\,\mu_F(F)=c\,\C_\H(F).$$
Thus, we have just proved the following proposition.
\begin{proposition}\label{chcgeq}
The capacities $\C_\Gamma$ and $\C_\H$ are equivalent on the family of the compact subsets of $S$.
\end{proposition}

\begin{corollary}\label{aHb}
For every $0<a\leq a_0$ and $b\geq b_0$ there exists a positive constant $c=c(a_0,b_0)$ such that
$$\frac{1}{c}\C_a(F)\leq\C_\H(F)\leq c\,\C_b(F)$$
for every compact set $F\subset S$.
\end{corollary}
The following proposition, with some interest in its own, it will be important in Section \ref{IBCCond}.

\begin{proposition}\label{cAtau}
Let $A$ be a compact subset of $\R^N$, $\tau\in]T_1,T_2[$, and $a>0$. There exists a positive constant $c$ independent of $A$ and $\tau$ such that
\begin{itemize}
\item[(i)] $\qquad\left|A\right|\leq c\,\C_a(A\times\{\tau\})$;
\item[(ii)] $\qquad\C_\H(A\times\{\tau\})\leq c\left|A\right|$.
\end{itemize}
\end{proposition}
\begin{proof}
Let us start by proving the first inequality. Let $\nu$ be the Lebesgue measure supported on $A\times\{\tau\}$. For any $(x,t)\in S$, we have
$$\G_a\ast\nu(x,t)=\int_A{\G_a(x,t,\xi,\tau)\,d\xi}\leq\int_{\R^N}{\G_a(x,t,\xi,\tau)\,d\xi}\leq\beta$$
for some positive constant $\beta$ (see \cite[Proposition 2.4]{LU}). Hence $G_a\ast\frac{\nu}{\beta}\leq 1$ in $S$, so that
$$\C_a(A\times\{\tau\})\geq\frac{\nu}{\beta}(A\times\{\tau\})=\frac{\left|A\right|}{\beta}.$$
Let us now prove the inequality in (ii). Let us consider a bounded open set $O\subset\R^N$ containing $A$, and denote by $\nu$ the Lebesgue measure supported on $O\times\{\tau\}$. We claim the following:
\begin{equation}\label{colaim}
\lim_{\substack{(x,t)\rightarrow(x_0,\tau)\\t>\tau}}{\Gamma\ast\nu(x,t)}=1\qquad\mbox{for every }x_0\in O.
\end{equation}
To prove this, let $x_0\in O$ be arbitrarily fixed. Let us consider two compactly supported continuous functions $\phi_1,\phi_2$ on $\R^N$ with $\phi_1(x_0)=\phi_2(x_0)=1$ such that $0\leq\phi_1\leq\chi_O\leq\phi_2\leq1$. Here $\chi_O$ denotes the characteristic function of $O$. For any $(x,t)\in\R^N\times]\tau,T_2[$ we have
$$\int_{\R^N}{\Gamma(x,t,\xi,\tau)\phi_1(\xi)\,d\xi}\leq\Gamma\ast\nu(x,t)=\int_{O}{\Gamma(x,t,\xi,\tau)\,d\xi}\leq\int_{\R^N}{\Gamma(x,t,\xi,\tau)\phi_2(\xi)\,d\xi}.$$
The hypothesis \eqref{H1} gives then \eqref{colaim}. Once we have proved the claim, let us pick $\mu\in\mathcal{M}^+(A\times\{\tau\})$ such that $\Gamma\ast\mu\leq 1$ in $S$. Then $u=\Gamma\ast\nu-\Gamma\ast\mu$ is $\H$-harmonic in $\R^N\times]\tau,T_2[$, and it satisfies $\liminf_{z\rightarrow(x_0,\tau)}{u(z)}\geq 0$ for all $x_0\in O$. This inequality also holds at any point $x_0\notin O$ since in this case we have
$$0\leq\limsup_{z\rightarrow(x_0,\tau)}{\Gamma\ast\mu(z)}\leq\Lambda\limsup_{z\rightarrow(x_0,\tau)}{\G_{a_0}\ast\mu(z)}=0$$
being $\sup_{\xi\in A}{\G_{a_0}(z,\xi,\tau)}\rightarrow0$ as $z\rightarrow(x_0,\tau)$. Moreover $u(z)\rightarrow 0$ as $z\rightarrow\infty$ (see \cite[Proposition 8.1]{LU}). Then, the comparison principle implies $u\geq 0$ in $\R^N\times]\tau,T_2[$, i.e.
$$\Gamma\ast\mu\leq \Gamma\ast\nu\qquad\mbox{in }\R^N\times]\tau,T_2[.$$
This inequality extends to $S\smallsetminus\left(A\times\{\tau\}\right)$ since $\Gamma\ast\mu=0$ in $\left(\R^N\times]T_1,\tau]\right)\smallsetminus\left(A\times\{\tau\}\right)$. Then, by Lemma \ref{commis} we get
$$\mu\left(A\times\{\tau\}\right)\leq \nu\left(A\times\{\tau\}\right)=\left|A\right|.$$
Since this holds true for every $\mu\in\mathcal{M}^+(A\times\{\tau\})$ with $\Gamma\ast\mu\leq 1$, we finally obtain
$$\C_\Gamma\left(A\times\{\tau\}\right)\leq \left|A\right|.$$
The desired inequality then follows from Proposition \ref{chcgeq}.
\end{proof}

\begin{corollary}
For every compact set $A\subset\R^N$, and $\tau\in]T_1,T_2[$, we have
$$\frac{1}{C}\left|A\right|\leq\C_\H(A\times\{\tau\})\leq C\left|A\right|,$$
with $C$ independent of $A$ and $\tau$.
\end{corollary}
\begin{proof}
We have just to put together Corollary \ref{aHb} and Proposition \ref{cAtau}.
\end{proof}

\section{Green estimates and crucial lemmas}\label{GL}

Let $z_0\in S$ be fixed. For simplicity of notations, we shall assume $z_0=(0,0)$. For any $0<\delta <1$, by the results proved in \cite[Section 6]{LU} we know that
\begin{itemize}
\item[\,]for every $0<r$ there exists an open set $D(r)$ satisfying $$B(0,\delta r)\subseteq D(r) \subseteq B(0,r)$$ with the property that the parabolic boundary points of the cylindrical domains $$D(r)\times ]a,b[, \qquad \mbox{for }T_1<a<b<T_2,$$ are $\H$-regular.
\end{itemize}

In what follows we are going to fix $\delta=\frac{1}{2}$. For every $M\geq e$, we let
$$C(M,r)=D\left(\sqrt{r\log{(M)}}\right)\times]-r,r[$$
and we denote by $G(M,r;z,\zeta)$ the Green function of $C(M,r)$ (see \cite[Section 7]{LU}). Then, for any $\zeta \in C(M,r)$,
$$G(M,r;z,\zeta)\rightarrow 0\,\,\, \mbox{as }z\rightarrow z_0,\qquad\forall\, z_0\in\de_p C(M,r).$$
We also know that $(z,\zeta)\mapsto G(M,r;z,\zeta)$ is nonnegative, lower semicontinuous, and, for any fixed $\zeta\in C(M,r)$, the function $z\mapsto G(M,r;z,\zeta)$ is smooth and $\H$-parabolic in $C(M,r)\smallsetminus\{\zeta\}$. Moreover we have
$$G(M,r;z,\zeta)>0\qquad\mbox{if }z=(z,t),\,\zeta=(\xi,\tau)\in C(M,r),\,t>\tau.$$
Therefore, with the terminology introduced in Section \ref{cap}, $G(M,r;\cdot,\cdot)$ is a kernel on $C(M,r)$. For every compact set $F\subset C(M,r)$ we will denote by $\C(M,r;F)$ the $G(M,r;\cdot,\cdot)$-capacity of $F$. Since $\Gamma$ is a kernel on $S$, then $\C_\Gamma(F)$ is well defined for any compact $F\subset S$. By $G\leq\Gamma$ we trivially have
$$\C_\Gamma(F)\leq \C(M,r;F)\qquad \forall F\subset C(M,r).$$
In what follows we shall use the following proposition.
\begin{proposition}\label{propeqpot}
Let $v$ be a $G(M,r;\cdot,\cdot)$-equilibrium potential of a compact set $F\subset C(M,r)$. Then
\begin{enumerate}
\item[(i)] $v\leq 1$ in $C(M,r)$;
\item[(ii)] $v$ is $\H$-parabolic in $C(M,r)\smallsetminus F$;
\item[(iii)] $v(z)\rightarrow 0$ as $z\rightarrow\zeta$ for every $\zeta\in \de_p C(M,r)$.
\end{enumerate}
\end{proposition}
\begin{proof}
By definition of equilibrium potential, there exists a nonnegative Radon measure $\nu$ supported in $F$ such that $v(z)=\int_F{G(M,r;z,\zeta)\,d\nu(\zeta)}$. Then the assertions follow by the properties of the Green function (see also \cite[Proposition 8.3]{LU}).
\end{proof}

We are interested in Gaussian bounds also for the Green kernels $G$. Since $G\leq\Gamma$, of course we have
\begin{equation}\label{uppbo}
G(M,r;z,\zeta)\le\Lambda\G_{a_0}(z,\zeta),\quad \forall z,\zeta\in C(M,r).
\end{equation}
Furthermore, in \cite[Section 7]{LU} it was obtained a Gaussian bound even from below. Here we will actually need something more precise.

\begin{lemma}\label{bzeroproprio}
There exists $\l_0\leq\frac{\delta}{4}=\frac{1}{8}$ depending just on $|\H|$ such that
\begin{equation}\label{lowbo}
G(M,r;z,\zeta)\geq\frac{1}{2}\Gamma(z,\zeta)\geq\frac{1}{2\Lambda}\G_{b_0}(z,\zeta),\quad \forall z,\zeta\in C(M,\lambda_0 r).
\end{equation}
\end{lemma}
\begin{proof} We will modify the arguments in \cite[Lemma 4.2 and Theorem 4.3]{BU}. Let us fix $z=(x,t),\zeta=(\xi,\tau) \in C(M,\lambda_0 r)$ with $t>\tau$, with $\l_0$ to be determined. We have
\begin{eqnarray*}
G(M,r;z,\zeta)&=&\Gamma(z,\zeta)-\int_{\de D(\sqrt{r\log{(M)}})\times[\tau,t]}{\Gamma(y,s,\zeta)\,d\mu_z(y,s)}\\
&=&\Gamma(z,\zeta)\left(1-\int_{\de D(\sqrt{r\log{(M)}})\times[\tau,t]}{\frac{\Gamma(y,s,\zeta)}{\Gamma(z,\zeta)}d\mu_z(y,s)}\right)
\end{eqnarray*}
for some nonnegative Radon measure $\mu_z$, which vanishes if $s>t$. Since $\mu_z(\de D(\sqrt{r\log{(M)}})\times[\tau,t])\leq 1$, it is enough to bound uniformly from above the ratio $\frac{\Gamma(y,s,\zeta)}{\Gamma(z,\zeta)}$ with something going to $0$ as $\l_0\rightarrow 0$. To this aim, by \cite[Proposition 2.2]{LU} and the doubling property we get
\begin{eqnarray*}
\frac{\Gamma(y,s,\zeta)}{\Gamma(z,\zeta)}&\leq& C\frac{\left|B(x,\sqrt{t-\tau})\right|}{\left|B(\xi,\sqrt{s-\tau})\right|}\exp{\left(-\frac{3a_0}{4}\frac{d^2(y,\xi)}{s-\tau}\right)}\exp{\left(b_0\frac{d^2(x,\xi)}{t-\tau}\right)}\\
&\leq&C\frac{\left|B(\xi,\sqrt{t-\tau}+d(x,\xi))\right|}{\left|B(\xi,\sqrt{t-\tau})\right|}\left(\frac{t-\tau}{s-\tau}\right)^{\frac{Q}{2}}\exp{\left(-\frac{3a_0}{4}\frac{d^2(y,\xi)}{s-\tau}\right)}\exp{\left(b_0\frac{d^2(x,\xi)}{t-\tau}\right)}\\
&\leq&C\left(\frac{\sqrt{t-\tau}+d(x,\xi)}{\sqrt{s-\tau}}\right)^{Q}\exp{\left(-\frac{a_0}{4}\frac{d^2(y,\xi)}{s-\tau}\right)}\exp{\left(b_0\frac{d^2(x,\xi)}{t-\tau}-\frac{a_0}{2}\frac{d^2(y,\xi)}{s-\tau}\right)},
\end{eqnarray*}
where we allowed the structural positive constant $C$ to change at every step. If $M=\max_{t>0}{t^{-\frac{Q}{2}}e^{-\frac{1}{t}}}$, we thus have
$$\frac{\Gamma(y,s,\zeta)}{\Gamma(z,\zeta)}\leq C\left(\frac{4}{a_0}\right)^{\frac{Q}{2}}M\left(\frac{\sqrt{t-\tau}+d(x,\xi)}{d(y,\xi)}\right)^{Q}\exp{\left(b_0\frac{d^2(x,\xi)}{t-\tau}-\frac{a_0}{2}\frac{d^2(y,\xi)}{s-\tau}\right)}.$$
By exploiting that $d(y,\xi)\geq d(y,0)-d(\xi,0)\geq \frac{1}{2}\sqrt{\delta r\log{M}}$, $s-\tau\leq t-\tau\leq 2\lambda_0 r$, $d(x,\xi)\leq 2\sqrt{\l_0r\log{M}}$, and $\log{M}\geq 1$ we obtain
$$\frac{\Gamma(y,s,\zeta)}{\Gamma(z,\zeta)}\leq C\l_0^{\frac{Q}{2}}\exp{\left(\left(4\l_0b_0-\frac{a_0}{8}\delta\right)\frac{r\log{M}}{t-\tau}\right)}$$
for a suitable positive structural constant $C$. Hence, if $\l_0\leq\frac{a_0}{b_0}\frac{\delta}{32}$, we have
$$\frac{\Gamma(y,s,\zeta)}{\Gamma(z,\zeta)}\leq C\l_0^{\frac{Q}{2}}.$$
Therefore, there exists a positive structural $\l_0$ such that $\frac{\Gamma(y,s,\zeta)}{\Gamma(z,\zeta)}\leq\frac{1}{2}$ for every $(y,s)\in\de D(\sqrt{r\log{(M)}})\times[\tau,t]$, and for every $z,\zeta\in C(M,\lambda_0 r)$ with $t>\tau$. This gives the assertion.
\end{proof}

In what follows we fix $0<\l\leq\l_0<\min{\{\frac{1}{e},\delta\}}$, with $\l_0$ as in $\eqref{lowbo}$. We are now going to prove two lemmas which will be very crucial in the sequel.

\begin{lemma}\label{CL1} Let $M\geq e$ and $0<r$ be fixed. Let $F$ be a compact set contained in
$$C(M,\l r)\cap\{(\xi,\tau)\in\R^{N+1}\,:\,\tau\leq -\l^2r\}.$$
Let us also denote by $v$ a $G(M,r;\cdot,\cdot)$-equilibrium potential of $F$ related to $C(M,r)$. Then there exists $p_0>0$ depending just on $|\H|,\l$ such that, if we denote by $p$ the first integer greater or equal than $p_0\log(M)$, we have
$$v(z)\geq\frac{1}{2}v(0)\qquad\forall\,z\in C(M,\l^pr).$$
\end{lemma}
\begin{proof} Let us denote $\Ch((x,t),r)=B(x,r)\times]t-r^2,t+r^2[$. Let $\mu$ be a $G(M,r;\cdot,\cdot)$-equilibrium measure of $F$ corresponding to $V$
\begin{equation}\label{vi}
v(z)=\int_F{G(M,r;z,\zeta)\,d\mu(\zeta)},\qquad z\in C(M,r).
\end{equation}
Let $\zeta\in F$ be arbitrarily fixed, and define
$$u(z)=\frac{G(M,r;z,\zeta)}{G(M,r;0,\zeta)},\qquad z\in C(M,\l^3 r).$$
We want to show the existence of a natural number $p\geq 3$, depending on $|\H|$ and $\log{(M)}$ as desired, such that
\begin{equation}\label{claimu}
u(z)\geq\frac{1}{2}\qquad \forall z\in C(M,\l^p r).
\end{equation}
By keeping in mind \eqref{vi}, this will prove the lemma. \\
The function $u$ is $\H$-parabolic in $C(M,\l^3 r)$. Moreover, if $\l^p\lgM<\frac{\delta^2}{2}\l^3=\frac{1}{8}\l^3$ holds true, we have the inclusions
$$C(M,\l^p r)\subseteq\Ch\left(0,\delta\sqrt{\frac{1}{2}\l^3r}\right)\subseteq\Ch\left(0,\delta\sqrt{\l^3r}\right)\subseteq C(M,\l^3r).$$
Thus, from the H\"older continuity of the $\H$-parabolic functions (see \cite[Theorem 7.2]{LU}, with the choice $\gamma=\frac{1}{\sqrt{2}}$) we have
\begin{equation}\label{hcont}
\left|u(z)-u(0)\right|\leq c\sup_{\Ch\left(0,\delta\sqrt{\l^3r}\right)}{\left|u\right|}\left(\frac{\hat{d}(z,0)}{\delta\sqrt{\l^3r}}\right)^\alpha\qquad\forall z\in C(M,\l^pr),
\end{equation}
where the positive constants $c$ and $\alpha$ ($\alpha<1$) depend just on $|\H|$. In order to estimate the supremum of $u$ we use the Gaussian bounds for $G$. Let $z=(x,t)\in \Ch\left(0,\delta\sqrt{\l^3r}\right)$ and denote $\zeta=(\xi,\tau)$. Then, by using \eqref{uppbo} and \eqref{lowbo}, we obtain
\begin{eqnarray}\label{bingu}
u(z)&\leq&2\Lambda^2\frac{\G_{a_0}(z,\zeta)}{\G_{b_0}(0,\zeta)}\leq2\Lambda^2\frac{\left|B(0,\sqrt{-\tau})\right|}{\left|B(x,\sqrt{t-\tau})\right|}\exp{\left(b_0\frac{d^2(0,\xi)}{-\tau}\right)}\nonumber\\
&\leq& 2\Lambda^2 M^{\frac{b_0}{\l}}\frac{\left|B(0,\sqrt{-\tau})\right|}{\left|B(x,\sqrt{t-\tau})\right|}.
\end{eqnarray}
On the other hand, we have
$$B\left(0,\sqrt{-\tau}\right)\subseteq B\left(x,d(x,0)+\sqrt{-\tau}\right)= B\left(x,\sqrt{t-\tau}\,\frac{d(x,0)+\sqrt{-\tau}}{\sqrt{t-\tau}}\right)\qquad\mbox{and}$$
$$\frac{d(x,0)+\sqrt{-\tau}}{\sqrt{t-\tau}}\leq\frac{\delta\sqrt{\l^3r}+\sqrt{\l r}}{\sqrt{(\l^2-\delta^2\l^3)r}}=\frac{\l+2}{\sqrt{\l}\sqrt{4-\l}}=c(\l).$$
Then, by using the doubling property in \eqref{bingu}, we obtain
$$\sup_{\Ch\left(0,\delta\sqrt{\l^3r}\right)}{u}\leq \tilde{c}(\l)M^{\frac{b_0}{\l}}$$
which implies by \eqref{hcont} that
$$\left|u(z)-u(0)\right|\leq c_1 M^{\frac{b_0}{\l}}\left(\frac{\l^pr\lgM}{\l^3r}\right)^\frac{\alpha}{2}\qquad\forall z\in C(M,\l^pr).$$
As a consequence, for every $z\in C(M,\l^pr)$, we get
$$u(z)= u(0)+(u(z)-u(0))\geq 1 - c_1 M^{\frac{b_0}{\l}}(\lgM)^\frac{\alpha}{2}\l^{(p-3)\frac{\alpha}{2}}\geq\frac{1}{2}:$$
the last inequality holds true if $p\geq p_0\lgM$ with a suitable choice of $p_0=p_0(\l)$ (independent of $M$). We want to remark that, in order to get \eqref{hcont}, we also assumed $\l^{p-3}\lgM<\frac{1}{8}$ which is satisfied with such a choice of $p$ (it would be satisfied even with a weaker $p\geq \tilde{p}_0\log{\lgM}$).
\end{proof}

\begin{lemma}\label{CL2}
Suppose we are given a sequence $\{r_k\}_{k\in\N}$ of positive real numbers such that $1\geq\l^p r_k\geq r_{k+1}$ for any $k\geq 1$, with $p$ the natural number of Lemma \ref{CL1}. Let $\{F_k\}$ be a sequence of compact sets such that
$$F_k\subset C(M,\l r_k)\cap\{(\xi,\tau)\in\R^{N+1}\,:\,\tau\leq -\l^2r_k\}\qquad\forall k\in\N.$$
Let us denote by $v_k$ a $G(M,r_k;\cdot,\cdot)$-equilibrium potential of $F_k$. For any $q\in\N$, let $V=V_q$ be the balayage potential of
$$F=\bigcup_{k=1}^q F_k.$$
Then, for every $k\in\N, k\leq q$, we have
$$1-V(z)\leq\exp{\left(-\frac{1}{2}\sum_{j=1}^kv_j(0)\right)}\qquad\forall z\in C(M,r_{k+1}).$$
\end{lemma}
\begin{proof}
Let us fix $q\in\N$ and denote for brevity $C_k=C(M,r_k)$. We split the proof in several steps.\\
\stepone\label{unonostar} Let us prove that
\begin{equation}\label{unostep}
V\geq v_1\qquad\mbox{in }C_1.
\end{equation}
Let $u\in\Phi_F$. Since $v_1$ is $\H$-parabolic in $\Omega_1=C_1\smallsetminus F_1$, the function $u-v_1$ is $\H$-superparabolic in $\Omega_1$. Moreover we have
$$\liminf_{\Omega_1\ni z\rightarrow\zeta}{u(z)-v_1(z)}\geq0\qquad\forall\zeta\in\de_p C_1\cup\de F_1,$$
since $u\geq 1$ on $F_1$ and $u\geq0$ everywhere, whereas $v_1\leq 1$ in $C_1$ and goes to $0$ on $\de_p C_1$ (see Proposition \ref{propeqpot}). Then, by the minimum principle for $\H$-superparabolic functions (see \cite[Proposition 3.10]{LU}), $u\geq v_1$ in $\Omega_1$. This inequality extends to all $C_1$ since $u\geq1\geq v_1$ on $F_1$. Considering that $u$ is an arbitrary function in $\Phi_F$, this implies $W_F\geq v_1$ in $C_1$. Hence
$$V(z)=\liminf_{\zeta\rightarrow z}W_F(\zeta)\geq\liminf_{\zeta\rightarrow z}v_1(\zeta)\geq v_1(z)\qquad \forall z\in C_1$$
and \eqref{unostep} is proved.\\
\steponep\label{unostar} Inequality \eqref{unostep} and Lemma \ref{CL1} imply
$$V(z)\geq\frac{1}{2}v_1(0)\qquad\forall z\in C_2$$
since $C(M,\l^p r_1)\supseteq C_2$.\\
\steptwo\label{duenostar} Let us now prove that
\begin{equation}\label{duestep}
\frac{V(z)-\frac{1}{2}v_1(0)}{1-\frac{1}{2}v_1(0)}\geq v_2(z)\qquad\forall z\in C_2.
\end{equation}
Let $u\in\Phi_F$. The function
$$w_2-v_2=\frac{u-\frac{1}{2}v_1(0)}{1-\frac{1}{2}v_1(0)}-v_2$$
is $\H$-superparabolic in $\Omega_2=C_2\smallsetminus F_2$. Moreover $w_2\geq 1$ on $F_2$, $v_2\leq 1$ in $C_2$, $v_2$ and goes to $0$ on $\de_p C_2$. By \hyperref[unostar]{$Step\,\, I^*$} we have also $w_2\geq 0$ in $C_2$. All these facts imply that $\liminf_{\Omega_2\ni z\rightarrow\zeta}{w_2(z)-v_2(z)}\geq0$ for all $\zeta\in\de_p C_2\cup\de F_2$. Thus, by just proceeding as in \hyperref[unonostar]{$Step\,\, I$}, we obtain \eqref{duestep}.\\
\steptwop\label{duestar} Inequality \eqref{duestep} and Lemma \ref{CL1} imply
$$\frac{V(z)-\frac{1}{2}v_1(0)}{1-\frac{1}{2}v_1(0)}\geq\frac{1}{2}v_2(0)\qquad\forall z\in C_3.$$
This inequality can be written as follows
\begin{equation}\label{dueprod}
V(z)\geq 1-\prod_{i=1}^2{\left(1-\frac{1}{2}v_i(0)\right)}\qquad\forall z\in C_3.
\end{equation}
\stepthree By using \eqref{dueprod} and arguing as in \hyperref[duenostar]{$Step\,\, II$}, we can prove that
$$\frac{V(z)-\left(1-\prod_{i=1}^2{\left(1-\frac{1}{2}v_i(0)\right)}\right)}{\prod_{i=1}^2{\left(1-\frac{1}{2}v_i(0)\right)}}\geq v_3(z)\qquad\forall z\in C_3.$$
This inequality and Lemma \ref{CL1} give
$$\frac{V(z)-\left(1-\prod_{i=1}^2{\left(1-\frac{1}{2}v_i(0)\right)}\right)}{\prod_{i=1}^2{\left(1-\frac{1}{2}v_i(0)\right)}}\geq \frac{1}{2}v_3(0)\qquad\forall z\in C_4$$
which can be written as follows
$$V(z)\geq 1-\prod_{i=1}^3{\left(1-\frac{1}{2}v_i(0)\right)}\qquad\forall z\in C_4.$$
\stepfour By iterating the previous procedure, for every $k\in\N$ with $k\leq q$ we get
$$V(z)\geq 1-\prod_{i=1}^k{\left(1-\frac{1}{2}v_i(0)\right)}\qquad\forall z\in C_{k+1}.$$
Therefore, for every $z\in C_{k+1}$,
$$1-V(z)\leq \exp{\left(\sum_{i=1}^k\log{\left(1-\frac{1}{2}v_i(0)\right)}\right)}\leq\exp{\left(-\frac{1}{2}\sum_{i=1}^kv_i(0)\right)}$$
by the elementary inequality $\log{(1-t)}\leq-t$ for $t<1$. The proof is thus complete.
\end{proof}

\section{Proof of the main results}\label{ZSchD}

Let $\O$ be a fixed bounded open set, with $\overline{\O}\subset S$, and let $z_0=(x_0,t_0)\in\de\O$. For
$\l \in ]0,1[$  and for any $h,k\in\N$, we define the compact sets
\begin{eqnarray}\label{omhk}
\Dhk&=&\left\{\zeta=(\xi,\tau)\in S\smallsetminus\O\,:\,\l^{k+1}\leq t_0-\tau\leq\l^k,\vphantom{\left(\frac{1}{\l}\right)^h}\right.\\
&\,&\quad\left.\exp{\left(\frac{d^2(x_0,\xi)}{t_0-\tau}\right)}\leq\left(\frac{1}{\l}\right)^h,\,\,\hat{d}(z_0,\zeta)\leq\sqrt{\l}\right\}. \nonumber
\end{eqnarray}

Moreover, for every $a,b>0$ and $s\in\R$, let us put
\begin{equation}\label{zabs}
z_a^b(\lambda; s)=\sum_{\N\ni k\leq s}\sum_{h=1}^{+\infty}\frac{\C_a\left(\Dhk\right)}{\left|B\left(x_0,\sqrt{\l^k}\right)\right|}\l^{bh},
\end{equation}
where we agree to let $z_a^b=0$ whenever the first summation is meaningless, i.e. for $s<1$. Finally, for every $z\in S$, let us define
$$Z_a^b(\lambda; z_0,z)=z_a^b\left(\l; \frac{\log{{\hat{d}}^2(z_0,z)}}{\log{\l}}\right).$$
The main aim of this section is to prove Theorem \ref{onedotwo}. Before starting the proof, some remarks are in order.
\begin{remark}\label{omegad} For every $\lambda \in]0,1[$, for every fixed $k \in \mathbb N$ and for every $a,b >0$, we have
$$ \sum_{h=1}^{\infty} \lambda^{b h} \C_a (\Omega_k^h(z_0,\lambda))\leq \sum_{h=1}^{\infty} \lambda^{b h} \C_a (\Dhk) \leq \frac{1}{1 - \lambda^b} \sum_{h=1}^{\infty} \lambda^{b h} \C_a (\Omega_k^h(z_0,\lambda)).$$
\end{remark}

\begin{remark}\label{zetamu} For every $\lambda, \mu \in ]0, 1[$ and for every $0<a<b$, there exists a positive constant $C$, only depending on $\lambda, \mu , a,b, Q $, such that
\begin{eqnarray}\label{zetalambdamu}
z_a^b (\lambda;s ) \leq  C\, \left( z_a^b (\mu;\sigma s ) + 1 \right) \quad  \mbox{for every } s,
\end{eqnarray}
where $\sigma = \frac{\log \lambda}{\log \mu}$.
\end{remark}
\noindent We postpone the proof of these remarks to Subsection \ref{app}. For our purposes it is now crucial to stress that from \eqref{zetalambdamu} it follows
\begin{eqnarray}\label{zetalambdamugrande}
Z_a^b(\lambda; z_0,z) \leq C \left( Z_a^b(\mu; z_0,z)  + 1\right) \quad \mbox{for every } z \in S.
\end{eqnarray}

\begin{proof}[Proof of Theorem \ref{onedotwo}] In the notations we  fixed above, we want to prove that
\begin{equation}\label{wmainz}
\W_{\rho}(z)\leq C\exp{\left(-\frac{1}{C}Z_a^b(\lambda; z_0,z)\right)}\qquad\mbox{for every }z\in S,
\end{equation}
for a suitable structural constant $C$. By Remark \ref{omegad}, it is equivalent to \eqref{wrhos}. Moreover, due to Remark \ref{zetamu}, it is not restrictive to assume $0 < \lambda\leq\lambda_0$, with $\lambda_0$ fixed in Lemma \ref{bzeroproprio}.

\noindent Let us fix $z_0=0$. For any $h,k\in\N$, we put
$$\Fhk=\left\{\zeta=(\xi,\tau)\in S\smallsetminus\O\,:\,\l^{k+1}\leq-\tau\leq\l^k,\,\,\,\exp{\left(\frac{d^2(0,\xi)}{-\tau}\right)}\leq\left(\frac{1}{\l}\right)^h\right\},$$
$G_k^h(\cdot,\cdot)$ the Green function of $C\left(\left(\frac{1}{\l}\right)^h,\l^{k-1}\right)$, and $v_k^h$ a $G_k^h$-potential of $\Fhk$. We remark that the compact set $\Fhk$ is compactly contained in $C\left(\left(\frac{1}{\l}\right)^h,\l^{k-1}\right)$. Moreover, for $l\in\N$, we put
$$\O_l=\left\{\zeta=(\xi,\tau)\in S\smallsetminus\O\,:\,\tau\leq0,\,\,\,\hat{d}(0,\zeta)\leq\l^\frac{l}{2}\right\}.$$
If $k\geq k(l,h)=l+\frac{\log{(1+h^2\log^2{\frac{1}{\l}})}}{2\log{\frac{1}{\l}}}$, then $\Fhk\subseteq\O_l$. We also note that, for any $k$, $\Fhk$ is actually contained in
$$C\left(\left(\frac{1}{\l}\right)^h,\l\cdot\l^{k-1}\right)\cap\left\{(\xi,\tau)\in\R^{N+1}\,:\,\tau\leq -\l^2\cdot\l^{k-1}\right\}.$$
Let $p\in\N$ be the one coming from Lemma \ref{CL1}, i.e. the smaller integer greater than $p_0h\log{\frac{1}{\l}}$. For every fixed $q\in\N$ there exists $j\in\{0,\ldots,p-1\}$ such that
$$\frac{1}{p}\sum_{k\geq k(l,h)}^{pq}v_k^h(0)\leq\sum_{\substack{i=0\\pi+j\geq k(l,h)}}^qv_{pi+j}^h(0).$$
Then, if $V_{\O_l}$ and $V$ denote respectively the balayage potentials of $\O_l$ and of $$\bigcup_{\substack{i=0\\pi+j\geq k(l,h)}}^q F_{pi+j}^h,$$
by the monotonicity of the balayage potential (see \cite[Proposition 4.2]{LU}) and by Lemma \ref{CL2} (with $r_i=\l^{pi+j-1}$) we have
\begin{equation}\label{duepunti}
1-V_{\O_l}(z)\leq 1- V(z)\leq \exp{\left(-\frac{1}{2}\sum_{\substack{i=0\\pi+j\geq k(l,h)}}^qv_{pi+j}^h(0)\right)}\leq \exp{\left(-\frac{1}{2p}\sum_{k\geq k(l,h)}^{pq}v_k^h(0)\right)}
\end{equation}
for every $z\in C\left(\left(\frac{1}{\l}\right)^h,\l^{p(q+1)+j-1}\right)\supset \B\left(0,\sqrt{\l^{p(q+2)}}\right)$.\\
Now, if $\nu_k^h$ denotes a $G_k^h$-equilibrium measure of $\Fhk$, we have for every $b>b_0$
\begin{eqnarray}\label{trepunti}
v_k^h(0)&=&\int_{\Fhk}{G_k^h(0,\zeta)}\,d\nu_k^h(\zeta)\geq \frac{1}{2\Lambda}\int_{\Fhk}{\G_{b_0}(0,\zeta)}\,d\nu_k^h(\zeta) \nonumber\\
&\geq& \frac{1}{2\Lambda}\frac{\l^{b_0h}}{\left|B(0,\l^{\frac{k}{2}})\right|}\nu_k^h\left(\Fhk\right)\geq c\frac{\C_a\left(\Fhk\right)}{\left|B(0,\l^{\frac{k}{2}})\right|}\l^{bh}\l^{-(b-b_0)h}.
\end{eqnarray}
by Lemma \ref{bzeroproprio} and the fact that $\C_{G_k^h}\geq \C_\Gamma \geq c \C_a$. Let us put $\alpha_k^h=\frac{\C_a\left(\Fhk\right)}{\left|B(0,\l^{\frac{k}{2}})\right|}\l^{bh}$. By using the estimate in \eqref{cipiuno} and the doubling property, we have
\begin{equation}\label{qpunti}
\alpha_k^h\leq C\frac{\left|B\left(0,\sqrt{\sqrt{2}\log{(\frac{1}{\l})}h\l^k}\right)\right|}{\left|B\left(0,\l^{\frac{k}{2}}\right)\right|}\l^{bh}
\leq C h^{\frac{Q}{2}}\l^{bh}.
\end{equation}
Inserting \eqref{trepunti} in \eqref{duepunti} and keeping in mind that $p\leq (p_0+1)h\log{\frac{1}{\l}}$, we get
$$-h\l^{(b-b_0)h}\log{\left(1-V_{\O_l}(z)\right)}\geq c\sum_{k\geq k(l,h)}^{pq}\alpha_k^h\qquad\mbox{if }\hat{d}(0,z)\leq\sqrt{\l^{p(q+2)}}.$$
On the other hand, by using \eqref{qpunti} we have
$$\sum_{k=1}^{k(l,h)}\alpha_k^h\leq C h^{\frac{Q}{2}}\l^{bh}k(l,h)=k^*(l,h).$$
Thus
\begin{equation}\label{cpunti}
-h\l^{(b-b_0)h}\log{\left(1-V_{\O_l}(z)\right)}\geq c\sum_{k=1}^{pq}\alpha_k^h-ck^*(l,h)\qquad\mbox{if }\hat{d}(0,z)\leq\sqrt{\l^{p(q+2)}}.
\end{equation}
Suppose $z\in S$ be such that $\hat{d}^2(0,z)\leq\l^{p(q+2)}$ and let $q$ be the minimum (if it exists) natural number satisfying this inequality. Then, letting $$\qz=\frac{\log{\left(\frac{1}{\d^2(0,z)}\right)}}{\log\frac{1}{\l}},$$ we have $p(q+2)\leq\qz\leq p(q+3)$ so that $pq\geq\qz-3p\geq\qz-\tilde{c}h$. Using this bound in \eqref{cpunti} we get
$$-h\l^{(b-b_0)h}\log{\left(1-V_{\O_l}(z)\right)}\geq c\sum_{k=1}^{\qz-\tilde{c}h}\alpha_k^h-ck^*(l,h)$$
for every $h\in\N$ and for every $z\in S$ such that $\qz\geq p(q+2)$ for at least one $q\in\N$, in particular for every $z\in S$ such that $\qz\geq3h (p_0+1)\log{\frac{1}{\l}}$. On the other hand
$$\sum_{\qz-\tilde{c}h\leq k\leq \qz}\alpha_k^h\leq C\tilde{c} h^{\frac{Q}{2}+1}\l^{bh},$$
hence, by letting $k^{**}(l,h)=k^*(l,h)+C\tilde{c} h^{\frac{Q}{2}+1}\l^{bh}$, we have
\begin{equation}\label{spunti}
-h\l^{(b-b_0)h}\log{\left(1-V_{\O_l}(z)\right)}\geq c\sum_{k=1}^{\qz}\alpha_k^h-ck^{**}(l,h)
\end{equation}
for every $z\in S$ and $h\in\N$ satisfying $3h (p_0+1)\log{\frac{1}{\l}}\leq\qz$. Let us now suppose $\qz>3(p_0+1)\log{\frac{1}{\l}}$. Then, inequality \eqref{spunti} holds true for any $h\in\N$ such that $h\leq h(z)=\frac{\qz}{3(p_0+1)\log{\frac{1}{\l}}}$. Thus, summing up in \eqref{spunti} with respect to $h$, we get
$$-\left(\sum_{h\leq h(z)}h\l^{(b-b_0)h}\right)\log{\left(1-V_{\O_l}(z)\right)}\geq c\sum_{h\leq h(z)}\sum_{k=1}^{\qz}\alpha_k^h-c\sum_{h\leq h(z)}k^{**}(l,h).$$
Therefore, since $\sum_{h=1}^{+\infty}h\l^{(b-b_0)h}\leq C_0<\infty$, $\sum_{h=1}^{+\infty}k^{**}(l,h)\leq C_0 l$, and
$$\sum_{h\geq h(z)}\sum_{k=1}^{\qz}\alpha_k^h\leq C\qz\sum_{h\geq h(z)}h^{\frac{Q}{2}}\l^{bh}\leq 3C(p_0+1)\log{\frac{1}{\l}}\sum_{h\geq h(z)}h^{\frac{Q}{2}+1}\l^{bh}\leq C_0$$
with $C_0$ independent of $z$ and $l$, we get
\begin{equation}\label{unti}
-\log{\left(1-V_{\O_l}(z)\right)}\geq\frac{c}{C_0}\sum_{h=1}^{+\infty}\sum_{k\leq\qz}\alpha_k^h-2c\cdot l
\end{equation}
for every $z\in S$ such that $\qz>3(p_0+1)\log{\frac{1}{\l}}$. On the other hand, if $\qz\leq3(p_0+1)\log{\frac{1}{\l}}$, we have
$$\sum_{h=1}^{+\infty}\sum_{k\leq\qz}\alpha_k^h\leq \sum_{h=1}^{+\infty}\sum_{k\leq3(p_0+1)\log{\frac{1}{\l}}}\alpha_k^h<\infty.$$
Thus, we can adjust the structural constants in \eqref{unti} in order that the relation \eqref{unti} holds true for every $z\in S$. Then, for some structural constant $C$, we finally have
$$1-V_{\O_l}(z)\leq\exp{(C\cdot l)}\exp{\left(-\frac{1}{C}\sum_{h=1}^{+\infty}\sum_{k\leq\qz}\frac{\C_a\left(\Fhk\right)}{\left|B(0,\l^{\frac{k}{2}})\right|}\l^{bh}\right)}\leq\exp{(C\cdot l)}\exp{\left(-\frac{1}{C}Z_a^b(z_0,z)\right)}$$
for every $z\in S$, and for every $l\in\N$. Thus, if we choose $\rho\in]0,1[$ such that $\rho<e^{-C}$, we have
$$\W(z)=\sum_{l=1}^{+\infty}\rho^l\left(1-V_{\O_l}(z)\right)\leq \exp{\left(-\frac{1}{C}Z_a^b(z_0,z)\right)}\sum_{l=1}^{+\infty}\left(\rho e^C\right)^l\leq C\exp{\left(-\frac{1}{C}Z_a^b(z_0,z)\right)}$$
for every $z\in S$ and the theorem is proved.
\end{proof}

As we remarked in the Introduction, the last theorem gives an estimate of the \emph{modulus of continuity} of the PWBB-solution to the Dirichlet problem only depending on the boundary datum $\phi$, on $\O$, and on the structural constants in $|\H|$. We are now going to see that it gives straightforwardly a Wiener-type $\H$-regularity test (Theorem \ref{mmmain}, part $(i)$). Furthermore, we have to prove the necessary counterpart for the $\H$-regularity (part $(ii)$).

\begin{proof}[Proof of Theorem \ref{mmmain}]
For the part $(i)$ we have just to observe that the hypothesis and \eqref{wrhos} imply that $\W(z)\rightarrow 0$ as $z\rightarrow z_0$. The $\H$-regularity of $z_0$ follows then by \eqref{charw}.\\
Let us turn to the proof of part $(ii)$. It follows from a result in \cite[Proposition 4.12]{LU}. As a matter of fact, we have
$$\left(\bigcup_{h,k\in\N}\Ohk\right)\cup\left(\left(\overline{B\left(x_0,\sqrt{\l}\right)}\times\{t=t_0\}\right)\smallsetminus\O\right)=\left(\overline{\B\left(z_0,\sqrt{\l}\right)}\cap\{t\leq t_0\}\right)\smallsetminus\O,$$
and $z_0\notin \Ohk$ for any $h,k\in\N$. Then, if $z_0$ is $\H$-regular,
\begin{equation}\label{seisom}
\sum_{h,k=1}^{+\infty}V_{h,k}(z_0)=+\infty
\end{equation}
where $V_{h,k}$ denotes the balayage potential of $\Ohk$. Let now $\mu_{h,k}$ be the balayage equilibrium measure of $\Ohk$. Then, by the very definition of $\Ohk$ in \eqref{omhktilde}, the doubling property, and by \cite[Proposition 8.3]{LU}, we get
\begin{eqnarray*}
V_{h,k}(z_0)&=&\int_{\Ohk}\hspace{-0.2cm}\Gamma(z_0,\zeta)\,d\mu_{h,k}(\zeta)\leq \Lambda\int_{\Ohk}\hspace{-0.2cm}\G_{a_0}(z_0,\zeta)\,d\mu_{h,k}(\zeta)\\
&\leq& \Lambda\frac{\l^{a_0(h-1)}}{\left|B\left(x_0,\sqrt{\l^{k+1}}\right)\right|}\mu_{h,k}\left(\Ohk\right)\leq C\frac{\l^{a_0h}}{\left|B\left(x_0,\l^{\frac{k}{2}}\right)\right|}\C_{b_0}\left(\Ohk\right),
\end{eqnarray*}
where in the last inequality we have exploited Corollary \ref{aHb}. The proof is then complete by inserting this relation in \eqref{seisom} and keeping also in mind that $\C_{b_0}\leq\C_b$ for any $b\geq b_0$.
\end{proof}

\subsection{Appendix to Section \ref{ZSchD}}\label{app}

Let us complete here the proofs of Remark \ref{omegad} and Remark \ref{zetamu}.

\begin{proof}[Proof of Remark \ref{omegad}] For every $h\in {\mathbb N}$,
$$
D_k^h(z_0,\lambda) = \cup_{j=1}^h\,\Omega_k^j(z_0,\lambda).
$$
Then, since ${\mathcal C}_a$ is subadditive,
\begin{eqnarray*}
\sum_{h=1}^{\infty} \lambda^{bh} \, {\mathcal C}_a (D_k^h(z_0,\lambda)) &\leq& \sum_{h=1}^{\infty} \lambda^{bh} \, \sum_{j=1}^h  \,{\mathcal C}_a (\Omega_k^j(z_0,\lambda))\\
= \sum_{j=1}^{\infty} \lambda^{bj}{\mathcal C}_a (\Omega_k^j(z_0,\lambda))\, \sum_{h=j}^{\infty} \lambda^{b(h-j)}&=& \frac{1}{1 - \lambda^b}\, \sum_{j=1}^{\infty} \lambda^{bj}{\mathcal C}_a (\Omega_k^j(z_0,\lambda)).
\end{eqnarray*}
The other inequality follows just by $\Ohk\subseteq\Dhk$ and the monotonicity of $\C_a$.
\end{proof}

\begin{proof}[Proof of Remark \ref{zetamu}] For every $k \in \mathbb N$ let us define
$$\sigma(k)=[\sigma k] = \mbox{integer part of } \sigma k.$$
Then, since $\l = \mu^{\sigma}$,
$$\lambda^k = \mu^{\sigma k} \leq \mu^{\sigma(k)}$$
and, letting $q = [\sigma] + 1$,
$$\lambda^{k+1} = \mu^{\sigma k + \sigma} \geq \mu^{\sigma(k) + 1 + \sigma} \geq \mu^{\sigma(k) + 1 + q}.$$
Summing up
$$\mu^{\sigma(k) + q + 1} \leq \lambda^{k+1} \leq \lambda^k \leq\mu^{\sigma(k)}.$$
Analogously
$$\left(\frac{1}{\mu}\right)^{\sigma(h)-q} \leq \left(\frac{1}{\lambda}\right)^{h-1} \leq \left(\frac{1}{\lambda}\right)^{h} \leq\left(\frac{1}{\mu}\right)^{\sigma(h) +1}$$
for every $h \in \mathbb N$. As a consequence, letting
$${\mathcal A}_{(k,h)}=\{ (i,j) \in \mathbb N \times \mathbb N : \sigma(k)\leq i \leq \sigma(k) + q, \,\, \sigma(h)-q+1  \leq j \leq \sigma(h) + 1\},$$
we have
$$\Omega_k^h(z_0, \lambda) \, \subset \, \bigcup_{(i,j) \in {\mathcal A}_{(k,h)}} \Omega_i^j(z_0, \mu).$$
Moreover, for every $(i,j) \in {\mathcal A}_{(k,h)}$,
$$\frac{\l^{bh}}{|B(x_0, \lambda^{\frac{k}{2}})|}\,\, \frac{|B(x_0, \mu^{\frac{i}{2}})|}{\mu^{bj}} \,\leq \, \left( \frac{\lambda^h}{\mu^{\sigma(h) + 1}} \right)^b \, \frac{|B(x_0, \mu^{\frac{\sigma(k)}{2}})|}{|B(x_0, \lambda^{\frac{k}{2}})|}
 \leq c_d\frac{1}{\mu^{b+\frac{Q}{2}}} = C_1$$
by the doubling condition (\hyperref[didue]{D2}). Therefore
$$\frac{\lambda^{bh}\, {\mathcal C}_a (\Omega_k^h(z_0, \lambda))}{|B(x_0, \lambda^{\frac{k}{2}})|} \leq  C_1 \sum_{(i,j) \in {\mathcal A}_{(k,h)}} \frac{\mu^{bj}\, {\mathcal C}_a (\Omega_i^j(z_0, \mu))}{|B(x_0, \mu^{\frac{i}{2}})|}.$$
To simplify the notation we denote by $c_{k,h}(\lambda)$ the term at the left hand side of this last inequality. Then
$$z_a^b(\lambda; s) = \sum_{k \leq s}\, \sum_{h \geq 1} c_{k,h}(\lambda)\, \leq \, C_1 \,  \sum_{k \leq s}\, \sum_{h \geq 1}\, \sum_{(i,j) \in {\mathcal A}_{(k,h)}} c_{i,j} (\mu).$$
On the other hand for a fixed $(i,j)$ we have
$$\sharp \left\{ (k,h) : (i,j) \in {\mathcal A}_{(k,h)}\vphantom{\sum}\right\}\, \leq \, \left( 1 + \frac{q+1}{\sigma}\right)^2,$$
whereas  $i \leq \sigma s+q+1$ if $(i,j)\in {\mathcal A}_{(k,h)}$ and $k\leq s$. Therefore
$$z_a^b(\lambda; s) \leq C_2 \sum_{i \leq \sigma s+q+1}\,\, \sum_{j \geq 1} c_{i,j}(\mu), \quad \mbox{where } C_2 = C_1 \left( 1 + \frac{q+1}{\sigma}\right)^2.$$
Hence
\begin{eqnarray}\label{zetapiccola}
z_a^b(\lambda; s)\, \leq \, C_2\, \left( z_a^b(\mu; \sigma s)\, +  \sum_ {\sigma s \leq i \leq \sigma s+q+1}\,\, \sum_{j \geq 1} c_{i,j}(\mu)  \right).
\end{eqnarray}
We now claim that
\begin{eqnarray}\label{stimacapacitaria}
c_{i,j}(\mu) \leq  \,\,\mu^{(b-a)j}.
\end{eqnarray}
Let us take for a moment this claim for granted. As a consequence, keeping in mind that $ b > a$,
$$\sum_ {\sigma s \leq i \leq \sigma s + q+1}\, \sum_{j \geq 1} c_{i,j}(\mu) \,\, \leq \,\, (q+2) \sum_{j \geq 1} \mu^{(b-a)j} = C_3.$$
Using this estimate in (\ref{zetapiccola}) we immediately obtain the inequality stated in Remark \ref{zetamu}. We are thus left with the proof of \eqref{stimacapacitaria}, i.e. with the proof of the following inequality
\begin{eqnarray}\label{stimacij}
\frac{ {\mathcal C}_a (\Omega_i^j(z_0,\mu))}{|B(x_0, \mu^{\frac{i}{2}})|} \, \leq \, \mu^{-aj}.
\end{eqnarray}
To this aim, we first remark that, for every $z=(x,t) \in \Omega_i^j(z_0, \mu)$, one has by definition
$$t_0 - t \leq \mu^i, \quad \mbox{and} \quad \exp \left(  -a \, \frac{d^2(x_0,x)}{t_0 - t}\right) \geq \mu^{aj}.$$
Therefore
$$\G_a (z_0,z) \geq \frac{1}{|B(x_0, \mu^{\frac{i}{2}})|}\, \mu^{aj}, \quad \forall\, z \in \Omega_i^j(z_0, \mu).$$
As a consequence, if $v$ is a $\G_a$- equilibrium potential of $\Omega_i^j(z_0, \mu)$ and $\nu$ is a corresponding equilibrium measure, we have
$$ 1 \geq v(z_0) = \int_{\Omega_i^j(z_0, \mu)} \G_a(z_0,z) d\nu(z) \geq \frac{1}{|B(x_0, \mu^{\frac{i}{2}})|} \mu^{aj}\, \nu(\Omega_i^j(z_0, \mu)).$$
Hence
$${\mathcal C}_a (\Omega_i^j(z_0,\mu))\, \leq \, |B(x_0, \mu^{\frac{i}{2}})|\, \mu^{-aj},$$
which is exactly (\ref{stimacij}).
\end{proof}

\section{Integral bound and cone condition}\label{IBCCond}

Let $\O$ be a bounded open set with $\overline{\O}\subset S$, and let $z_0=(x_0,t_0)\in\de\O$. For $\l\leq\l_0$ and such that $T_1<t_0-\l$, let us define, for $\rho>1$ and $t_0-\l\leq\tau<t_0$,
$$\Orol=\left\{z=(x,t)\in S\smallsetminus\O\,:\d(z_0,z)\leq \sqrt{\l},\quad\exp{\left(\frac{d^2(x_0,x)}{t_0-t}\right)}\leq\rho\right\},$$
$$\erolt=\Orol\cap\{t=\tau\},\quad\mbox{and}\quad\mrolt=\left|\erolt\right|.$$
We would like to prove now Theorem \ref{intnotcap}, i.e. an estimate of $\W$ in terms of $m_\l(\cdot,\cdot)$ or more precisely in terms of the following function
$$M_\l(z_0,z)=\int_{\d^2(z_0,z)}^{\l}{\int_1^{+\infty}{\frac{m_\l\left(\rho,t_0-\eta\right)}{\left|B\left(x_0,\sqrt{\eta}\right)\right|}}\frac{d\rho}{\rho^{1+b}}\frac{d\eta}{\eta}}.$$

\begin{proof}[Proof of Theorem \ref{intnotcap}] In our notations, we want to prove the following
\begin{equation}\label{zibdue}
\W_{\rho}(z)\leq C\exp{\left(-\frac{1}{C}M_\l(z_0,z)\right)}\qquad\mbox{for every }z\in S,
\end{equation}
for $\rho$ small enough and a suitable $C$.\\
First of all, keeping in mind the definition of $\Dhk$, we have
$$\Dhk\supseteq E_\l\left(\left(\frac{1}{\l}\right)^h,\tau\right),\qquad\mbox{for }\l^{k+1}\leq t_0-\tau\leq\l^k.$$
Let us fix $0<a\leq a_0$ and $b>b_0$, where $a_0$ and $b_0$ are the positive constants in (\hyperref[bounds]{H}). Then, by Proposition \ref{cAtau},
$$\C_a(\Dhk)\geq\frac{1}{C}m_\l\left(\left(\frac{1}{\l}\right)^h,\tau\right)\qquad\mbox{for every }\l^{k+1}\leq t_0-\tau\leq\l^k.$$
Thus, for $s\geq 1$,
\begin{eqnarray*}
z_a^b(\l;s)&\geq&\frac{1}{C}\sum_{h=1}^{+\infty}\l^{hb}\sum_{k\leq s}\frac{1}{\left|B\left(x_0,\sqrt{\l^k}\right)\right|}\max_{\l^{k+1}\leq t_0-\tau\leq\l^k}{m_\l\left(\left(\frac{1}{\l}\right)^h,\tau\right)}\\
&\geq&\frac{\l^{\frac{Q}{2}}}{C}\sum_{h=1}^{+\infty}\l^{hb}\sum_{k\leq s}\int_k^{k+1}{\frac{m_\l\left(\left(\frac{1}{\l}\right)^h,t_0-\l^\sigma\right)}{\left|B\left(x_0,\l^{\frac{\sigma}{2}}\right)\right|}\,d\sigma}\\
&\geq&\frac{\l^{\frac{Q}{2}}}{C}\sum_{h=1}^{+\infty}\l^{hb}\int_1^{s}{\frac{m_\l\left(\left(\frac{1}{\l}\right)^h,t_0-\l^\sigma\right)}{\left|B\left(x_0,\l^{\frac{\sigma}{2}}\right)\right|}\,d\sigma}\\
&=&\frac{\l^{\frac{Q}{2}}}{C\log{\frac{1}{\l}}}\sum_{h=1}^{+\infty}\l^{hb}\int_{\l^s}^{\l}{\frac{m_\l\left(\left(\frac{1}{\l}\right)^h,t_0-\eta\right)}{\left|B\left(x_0,\sqrt{\eta}\right)\right|}\frac{d\eta}{\eta}}\\
&\geq&\frac{\l^{\frac{Q}{2}}}{C\log{\frac{1}{\l}}}\int_{\l^s}^{\l}{\frac{1}{\left|B\left(x_0,\sqrt{\eta}\right)\right|}\left(\sum_{h=2}^{+\infty}\int_{h-1}^h{\l^{bh}m_\l\left(\left(\frac{1}{\l}\right)^h,t_0-\eta\right)}\,dr\right)\frac{d\eta}{\eta}}\\
&\geq&\frac{\l^{\frac{Q}{2}+b}}{C\log{\frac{1}{\l}}}\int_{\l^s}^{\l}{\frac{1}{\left|B\left(x_0,\sqrt{\eta}\right)\right|}\left(\int_1^{+\infty}{\l^{br}m_\l\left(\left(\frac{1}{\l}\right)^r,t_0-\eta\right)}\,dr\right)\frac{d\eta}{\eta}}\\
&=&\frac{\l^{\frac{Q}{2}+b}}{C\log^2{\frac{1}{\l}}}\int_{\l^s}^{\l}{\frac{1}{\left|B\left(x_0,\sqrt{\eta}\right)\right|}\left(\int_{\frac{1}{\l}}^{+\infty}{m_\l\left(\rho,t_0-\eta\right)}\frac{d\rho}{\rho^{1+b}}\right)\frac{d\eta}{\eta}}.
\end{eqnarray*}
Then, we have proved the inequality
\begin{equation}\label{zibuno}
z_a^b(\l;s)\geq c\int_{\l^s}^{\l}{\int_{\frac{1}{\l}}^{+\infty}{\frac{m_\l\left(\rho,t_0-\eta\right)}{\left|B\left(x_0,\sqrt{\eta}\right)\right|}}\frac{d\rho}{\rho^{1+b}}\frac{d\eta}{\eta}}.
\end{equation}
On the other hand, since
$$\frac{m_\l\left(\rho,t_0-\eta\right)}{\left|B\left(x_0,\sqrt{\eta}\right)\right|}\leq\frac{\left|B\left(x_0,\sqrt{\eta\log{\rho}}\right)\right|}{\left|B\left(x_0,\sqrt{\eta}\right)\right|}\leq 1+c_d(\log{\rho})^{\frac{Q}{2}},$$
the integral at the right hand side of \eqref{zibuno} can be estimated from below with
$$\int_{\l^s}^{\l}{\int_1^{+\infty}{\frac{m_\l\left(\rho,t_0-\eta\right)}{\left|B\left(x_0,\sqrt{\eta}\right)\right|}}\frac{d\rho}{\rho^{1+b}}\frac{d\eta}{\eta}}-C.$$
As a consequence
$$z_a^b(\l;s)\geq c\left(\int_{\l^s}^{\l}{\int_1^{+\infty}{\frac{m_\l\left(\rho,t_0-\eta\right)}{\left|B\left(x_0,\sqrt{\eta}\right)\right|}}\frac{d\rho}{\rho^{1+b}}\frac{d\eta}{\eta}}-C\right)\qquad\mbox{for }s\geq 1;$$
whereas for $s<1$ such inequality is satisfied by definition. Therefore, by \eqref{wmainz} we get \eqref{zibdue}.\\
For the last statement in Theorem \ref{intnotcap} we have just to observe that the divergence of the integral implies that $\W(z)\rightarrow 0$ as $z\rightarrow z_0$. The $\H$-regularity of $z_0$ follows then by \eqref{charw}.
\end{proof}

Let us now recall explicitly the definition of $d$-cone condition.
\begin{definition}[Exterior $d$-cone condition]\label{D4.8}
We say that $\Omega$ satisfies the exterior $d$-cone condition at a point $z_0=(x_0,t_0)\in\de \Omega$ if there exist $M_0,r_0,\theta>0$ such that
\begin{equation}\label{4.4}
|\{x\in \overline{B(x_0,M_0r)}\,:\,(x,t_0-r^2)\not\in\Omega\}|\ge\theta|B(x_0,M_0r)|
\end{equation}
for every $0<r\le r_0$.
\end{definition}

We want to prove that, if $\O$ satisfies this condition at $z_0\in\de\O$, then $\W$ is H\"older continuous at $z=z_0$. This will give a quantitative version of the $\H$-regularity of $z_0$, already proved in \cite[Theorem 4.11]{LU}.

\begin{theorem}\label{cconduno}
Assume $\O$ satisfies the exterior $d$-cone condition at $z_0=(x_0,t_0)\in\de\O$. Then there exist $C$ and $\alpha$ such that
$$\W(z)\leq C\left(\d(z_0,z)\right)^\alpha.$$
The positive constants $C$ and $\alpha$ ($\alpha\leq 1$) depend only on the parameters $M_0, r_0, \theta$ in Definition \ref{D4.8}, and on $|\H|$.
\end{theorem}
\begin{proof}
Assume there exist $M_0,r_0,\theta>0$ as in Definition \ref{D4.8}. Then, there exists $\bar{\rho}>1$ ($\bar{\rho}=\exp{(M_0^2)}$) such that
$$m_\l(\bar{\rho},\tau)\geq\theta\left|B\left(x_0,\sqrt{\log{(\bar{\rho})}(t_0-\tau)}\right)\right|$$
for every $0<t_0-\tau\leq\delta(\l)=\min{\left\{r_0^2,\frac{\l}{\sqrt{1+M_0^4}}\right\}}$. By using the doubling property, if $\d^2(z_0,z)\leq\delta(\l)$ we have
\begin{eqnarray*}
M_\l(z_0,z)&\geq&\int_{\d^2(z_0,z)}^{\delta(\l)}{\frac{1}{\left|B\left(x_0,\sqrt{\eta}\right)\right|}\left(\int_{\bar{\rho}}^{+\infty}{m_\l\left(\rho,t_0-\eta\right)\frac{d\rho}{\rho^{1+b}}}\right)\frac{d\eta}{\eta}}\\
&\geq&\int_{\d^2(z_0,z)}^{\delta(\l)}{\frac{m_\l\left(\bar{\rho},t_0-\eta\right)}{\left|B\left(x_0,\sqrt{\eta}\right)\right|}\left(\int_{\bar{\rho}}^{+\infty}{\frac{d\rho}{\rho^{1+b}}}\right)\frac{d\eta}{\eta}}\\
&\geq&c\int_{\d^2(z_0,z)}^{\delta(\l)}{\frac{d\eta}{\eta}}=c\left(\log{\frac{1}{\d^2(z_0,z)}}-\log{\frac{1}{\delta(\l)}}\right).
\end{eqnarray*}
Inserting this inequality in \eqref{zibdue}, we immediately obtain the assertion in the case $\d^2(z_0,z)\leq\delta(\l)$. The remaining case follows just from the boundedness of $\W$.
\end{proof}

\begin{corollary}\label{ncor}
Assume the exterior $d$-cone condition holds at $z_0$. Then, for every $\phi\in C(\de\O,\R)$, we  have
$$ |H_\phi^\Omega(z)-\phi(z_0)|\leq \tilde{\phi}\left(z_0,C\left(\d(z_0,z)\right)^\alpha\right)\qquad\forall z\in\Omega,$$
where $C,\alpha$ are the constants in Theorem \ref{cconduno}.
\end{corollary}
\begin{proof}
The assertion follows by the last theorem and \eqref{5.9}, by keeping in mind the monotonicity of $\tilde{\phi}$ in the second variable.
\end{proof}

We close the paper by completing the proof of Theorem \ref{conooo} and by giving an application to cylindrical domains.
\begin{proof}[Proof of Theorem \ref{conooo}]
We can deduce it by putting together Corollary \ref{ncor} and \cite[Proposition 5.7]{LU}. We explicitly remark that in \cite[Proposition 5.7]{LU} it was supposed the validity of a reverse-doubling property for $d$. This holds true by using the properties (\hyperref[diuno]{D1})--(\hyperref[ditre]{D3}) and by arguing as in \cite[Proposition 2.9 - Lemma 2.11]{DGL} (see also \cite{FGW}).
\end{proof}

\begin{corollary}
Let $\O=D\times ]t_1,t_2[$ be a cylindrical domain, with $\overline{\O}\subset S$. Assume $D$ satisfies the following condition at some point $x_0\in\de D$
\begin{enumerate}
  \item[\,]there exist $r_0,\theta>0$ such that $\quad|B(x_0,r)\smallsetminus D|\ge\theta|B(x_0,r)|\,\,$ for every $0<r\le r_0$.
  \end{enumerate}
Then, at every point $(x_0,t_0)$ with $t_1\leq t_0\leq t_2$, the conclusions of Theorem \ref{conooo} hold true.
\end{corollary}
\begin{proof}
We have just to recall \cite[Proposition 6.1]{LU} which provides the validity of the $d$-cone condition at such $z_0$.
\end{proof}

\bibliographystyle{amsplain}

\end{document}